\documentclass[12pt,twoside
]{article}

\usepackage{graphpap}
\usepackage{epsfig}

\usepackage{amscd}

\usepackage{amsmath}
\usepackage{amssymb}
\usepackage{amsthm}

\makeatletter
\def\@seccntformat#1{\csname the#1\endcsname{.}\hskip .5em}

\renewcommand{\section}{\@startsection {section}{1}{\z@}%
                                {-4.5ex \@plus -1ex \@minus-.2ex}%
                                   {2.3ex \@plus.2ex}%
                               {\reset@font\Large\scshape\centering}}

	\renewenvironment{proof}[1][\proofname]{\par					   
	  \normalfont												   
	  \topsep6\p@\@plus6\p@	\trivlist							   
	  \item[\hskip\labelsep\bfseries							   
		#1\@addpunct{.}]\ignorespaces							   
	}{%
	  \qed\endtrivlist											   
	}		

\makeatother

\normalbaselines

\newtheorem{thm}{Theorem}[section]
\newtheorem{lm}[thm]{Lemma}

\theoremstyle{definition}

\theoremstyle{remark}

\numberwithin{equation}{section}

\newcommand{\f}{\varphi}

\makeatletter

\def\multilimits@{\bgroup
  \Let@
  \restore@math@cr
  \default@tag
 \baselineskip\fontdimen10 \scriptfont\tw@
 \advance\baselineskip\fontdimen12 \scriptfont\tw@
 \lineskip\thr@@\fontdimen8 \scriptfont\thr@@
 \lineskiplimit\lineskip
 \vbox\bgroup\ialign\bgroup\hfil$\m@th\scriptstyle{##}$\hfil\crcr}  
\def\Sb{_\multilimits@}
\def\Sp{^\multilimits@}
\def\endSb{\crcr\egroup\egroup\egroup}

\makeatother

\newcommand{\supp}{\operatorname{supp}}

\newcommand{\al}{\alpha}

\newcommand{\C}{\mathbb{C}}

\newcommand{\e}{\varepsilon}

\newcommand{\dist}{\operatorname{dist}}

\newcommand{\pd}{\partial}
\newcommand{\Om}{\Omega}
\newcommand{\om}{\omega}
\newcommand{\bP}{\mathbb{P}}
\newcommand{\E}{\mathbb{E}}
\newcommand{\F}{\Phi}

\newcommand{\R}{\mathbb{R}^n}

\newcommand{\SSS}{\mathcal{S}}

{\end{list}}

\begin{document}

\title{Two weight estimate for the Hilbert transform and corona decomposition for non-doubling measures}
\author{ F. Nazarov, \thanks{Department of Mathematics, Michigan State
University, East Lansing, MI 48824, USA}
S. Treil\thanks{Department of Mathematics, Brown University, Providence, RI    USA}
 \,\,and A. Volberg
\thanks{Department of Mathematics, Michigan State University, East
Lansing, MI 48824, USA; 
}{ }\thanks{All  authors are  supported by the NSF grant DMS 
0200713}
{ } 
{ }}
\date{}
\maketitle

%



\tableofcontents

 \pagestyle{headings}
 \renewcommand{\sectionmark}[1]%
        {\markboth{}{\hfill\thesection{.}\ #1\hfill}}

         \renewcommand{\subsectionmark}[1]{\relax}
\setcounter{section}{0}
\section{Introduction}
\label{Intr}

We consider here a  problem of finding necessary and sufficient conditions for the boundedness of two
weight Calder\'on-Zygmund operators. We give such necessary and sufficient conditions in very natural terms, if the
operator is the Hilbert transform, and the weights satisfy some very natural condition. This condition might even happen to be necessary for the two weight boundedness of the Hilbert transform. This, of course, would be wonderful, because if this really happens,
 our result would give necessary and sufficient condition for the two weight boundedness of the Hilbert transform for two arbitrary weights (measures). However we cannot either prove or disprove the necessity of our condition (called the {\it pivotal condition}  in the text below). But we definitely know that it is satisfied for doubling measures. Thus we reprove the results from \cite{NTV7}.  We also indicate some other situations when our pivotal condition is automatically satisfied and is easily verified.

Our necessary and sufficient conditions of the boundedness seem to be quite natural. Actually in a one weight (one measure) case they become a famous $T1$ conditions under small disguise. We discuss why our conditions are exactly the correct   generalization of $T1$ conditions of David--Journ\'e from one measure case to two measure case a bit later. Now we just want to warn the reader that even in the one measure case
considered by  David--Journ\'e \cite{DJ1}, \cite{DJ2} , they really used that their measure is Lebesgue measure.
More general $T1$ theorems were proved by Christ \cite{Ch}, these were again one measure $T1$ theorems, but now the measure under consideration was allowed to be arbitrary measure with doubling condition (homogeneous measure by the widespread terminology). It took considerable efforts to get rid of the last assumption.
Now $T1$ theorems for one measure exist, and they do not use homogeneity. This is the scope of non-homogeneous Harmonic Analysis, and we refer the reader to \cite{NTV1}--\cite{NTV6} and to \cite{To1}, \cite{To2}.

The reader will see what we understand by $T1$ theorem for two measure a bit later, but
now we can already say that this will be Sawyer type test conditions. In other words, our $T1$ generalization
amounts to just testing the operator $T$ (and its adjoint) on characteristic functions of cubes (intervals) exactly as this has been done by Sawyer in the series of works \cite{Saw1}--\cite{Saw2},  which appeared at approximately the same time as David-Journ\'e's $T1$ theory. 

Of course, the difference between Sawyer's works and David--Journ\'e's works is that he considered two weight situation, and they considered one weight situation
(actually Lebesgue measure situation), on the other hand David-Journ\'e considered {\it singular} operators, while operators considered by Sawyer were not singular, these were the operators with positive kernels.
But strangely enough, to the best of our knowledge, it was not a common place that $T1$ conditions are {\it identical} to Sawyer's test conditions! 

Maybe nobody (as far as we can say) made a point by saying that these are two equivalent assumptions because they were applied in different situations.  But  let us confirm that the test conditions of Sawyer and $T1$ conditions of David--Journ\'e are actually identical. Suppose $T$ is an operator with
a Calder\'on--Zygmund kernel $k(x,y)$ (we will assume that the kernel is antisymmetric a bit later, this is only for the sake of brevity of reasoning). We have two seemingly different claims:
\begin{itemize}
\item $T1\in BMO$,
\item $\forall Q\,\,\|T\chi_Q\|_{L^2(Q)}^2 \leq c\,|Q|\,.$
\end{itemize}
The first claim is of course David--Journ\'e's $T1$--theorem assumption, the second one is one of several possible Sawyer's test type conditions. Of course they are equivalent and in a trivial way.  In fact, let us assume the Sawyer's test condition. To prove that $T1\in BMO$ we need to fix any cube $Q$, then we consider $1= \chi_{\R\setminus 2Q} + \chi_{2Q}$ decomposition and apply $T$ to it.
By the  Calder\'on--Zygmund property of the kernel one immediately gets (see \cite{DJ1}) that function
$f_1:= T\chi_{\R\setminus 2Q}$ satisfies $\int_Q |f_1 - c_Q|^2 dx \leq c |Q|$ for a certain constant $c_Q$. But the Sawyer's test condition
gives that $\int_Q |T\chi_{2Q}|^2 dx \leq c|2Q|\leq  C |Q|$.  Then immediately $\int_Q |T1 -c_Q|^2 dx \leq C|Q|$, and this means $T1 \in BMO$.

On the other hand, if we assume first David--Journ\'e's condition $T1 \in BMO$ then the same decomposition brings us the claim that $\int_Q  |T\chi_{2Q} - c_Q|^2 dx  \leq C |Q|$ for a  constant $c _Q=\frac1{|Q|}\int_Q T\chi_{2Q}$. Then one can estimate this constant easily. This is especially easy for Calder\'on--Zygmund operators with antisymmetric kernel ($k(x,y) = - k(y,x)$), which is the leading interesting case anyway. So let us consider only antisymmetric $T$'s. And for them obviously $|c_Q|= |\frac1{|Q|}\int_{\R}\chi_Q T\chi_{2Q\setminus Q} dx|$. We put the absolute value inside the integral now and  use the roughest possible estimate of Calder\'on--Zygmund  kernel $k$, $|k(x,y)| \leq \frac{C}{|x-y|^n}$, which  gives $|c_Q| \leq C$. Then $\int_Q |T\chi_{2Q} - c_Q|^2 dx  \leq C |Q|$ and $|c_Q| \leq C$ give us  $\int_Q |T\chi_{2Q} |^2 dx  \leq C |Q|$.
Now to say $\int_{Q} |T\chi_{Q}|^2 dx \leq C |Q|$
we need the estimate $\int_Q |T\chi_{2Q\setminus Q}|^2 dx \leq C |Q|$, which is immediate again from the same
roughest estimate of Calder\'on--Zygmund kernel.

Our paper is devoted to two weight case. And we believe that Sawyer's test condition (which we just showed to be equal to $T1$ condition of David--Journ\'e  in a one weight case) gives the correct generalization of $T1$ theorems to the two weight (two measures) situation. So for us the two weight $T1$ theorem is just the result, which says that a singular operator is bounded  from one $L^2$ to {\it another} $L^2$ if and only if being tested on characteristic functions it is uniformly bounded---exactly as in some Sawyer's test conditions.

For a certain model Calder\'on-Zygmund operators 
we prove exactly this type of $T1$ theorem in  \cite{NTV-2w}, \cite{NTV6}.
These are certain dyadic Calder\'on-Zygmund operators, 
and we give in \cite{NTV-2w}, \cite{NTV6} a necessary and sufficient conditions on weights, without assuming anything, for these operators to be bounded between two different weighted $L^2$ spaces. In particular, in \cite{NTV-2w}, \cite{NTV6} we treated a two weight $T1$ theorem for  the so-called Martingale Transforms, which are sometimes considered as a dyadic version of the Hilbert transform.
It is known that the Hilbert Transform and the Martingale Transform are very closely related, see, for example, \cite{Bo1},
\cite{Bu1}.

Notice that the  two weight problem for singular operators seemed to be extremely difficult, adequate tools seemed to be not
available. The theory of nonhomogeneous
Calder\'on-Zygmund operators, as we will see below, at least gives a considerable hope to understand such two weight
problems. 

The interest to two weight problems for singular operators naturally appears from an attempt to understand
when the operator in the Hilbert space has an unconditional spectral decomposition. Due to Wermer \cite{W}, the following
rigidity claim holds: this unconditional spectral decomposition exists for $T$ if and only if $T=S^{-1} N S$, where $N$ is a
normal operator, and $S$ is invertible (similarity). This similarity to normal operator question got a large attention
recently for different classes of $T$. We mention here \cite{BeNi}, \cite{NiT}, \cite{Ku}, \cite{Ka},  \cite{N}, \cite{NV}. If $T$ is a small perturbation of
a unitary operator (even a rank one perturbation), then in general the criteria of similarity with a normal operator is
more or less totally open.  Even if $T$ is a contraction, the relation between the spectral data of $U$ and $N$ is very
subtle in general. This kind of questions very fast become related to two weight problems for the Cauchy transform, as
illustrated by \cite{NiT}. For example, \cite{NiT} is based on a remarkable example of Fedor Nazarov, which says that
Hunt-Muckenhoupt-Wheeden criterion for one weight boundedness of the Hilbert transform is not applicable in two weight
situation. The reader will find more details in \cite{N}, \cite{NV} and in Section \ref{TwoWeightPrelim},
\ref{necessityTW}.

An unexpected application of two weight Hilbert Transform was awaiting in a problem from spectral theory of almost periodic Jacobi matrices, see \cite{VoYu},\cite{PVoYu}. In these papers the singular continuous spectrum of a wide and natural class of Jacobi matrices got related to the properties of a certain Gehktman--Faybusovich flow (see \cite{GF}), which is not unlike a well-known Toda flow of Jacobi matrices. In its turn the uniform boundedness  in this flow turns out to be exactly equivalent to a certain two weight Hilbert transform problem. 

\vspace{.2in}

Finally, let us explain what is the main difficulty of the theory of nonhomogeneous
Calder\'on-Zygmund operators. Roughly speaking the difficulty appears as a result of a certain {\it degeneracy} in the
operator. We can evoke the vague analogy with subellipticity in PDE. In our case, the degeneracy appears not in the kernel
of the operator--the kernel is a classical Calder\'on-Zygmund kernel--but in underlying measure. To illustrate, what kind of
difficulty persistently appears let us think that we need to estimate the quantity
\begin{equation}
\label{etoile}
I:= |\int_Q\int_R k(x,y)f(x)g(y)\,d\mu(x)\,d\mu(y)|\,.
\end{equation}

 Three possibilities can logically occur: 1) to estimate $k$ in $L^{\infty}$ (may be after using some sort of cancellation),
and to estimate $f$ in $L^1(\mu)$, $g$ in $L^1(\mu)$; 2) to estimate $k$ in $L^{1}L^{\infty}$ (this is a mixed norm, $L^1$
in the first variable, $L^{\infty}$ in the second one), and to estimate
$f$ in
$L^{\infty}(\mu)$, $g$ in $L^1(\mu)$; 3) to estimate $k$ in $L^{1}$, and to estimate
$f$ in
$L^{\infty}(\mu)$, $g$ in $L^{\infty}$.

In the first case no difficulty appears. We need to bound $I$ in \eqref{etoile} by $\|f||_{L^2}\|g\|_{L^2}$, and this
is not a problem, by Cauchy inequality $\|f\|_{L^1} \leq \mu(Q)^{1/2} \|f\|_{L^2}, \|g\|_{L^1} \leq \mu(R)^{1/2}
\|g\|_{L^2}$.

Suppose we want to repeat something like that in the second case. First of all $L^{\infty}$ norm cannot be estimated by
$L^2$ one. But this is not the difficulty (strangely), because in expression $I$ usually $f,g$ are very simple, basically
constant functions on $Q,R$. In this case we have the desired estimates: $\|f\|_{L^{\infty}} \leq \mu(Q)^{-1/2} \|f\|_{L^2},
\|g\|_{L^1} \leq \mu(R)^{1/2}
\|g\|_{L^2}$. Subsequently, we get the expression $\frac{\mu(R)^{1/2}}{\mu(Q)^{1/2}}$. This is a not so nice an expression
because measure of a (small) set $Q$ stands in the denominator. For good measures (for example for Lebesgue measure) we
have a control
of these ``small denominators". But for an arbitrary measure, the denominator can be arbitrarily small, or even vanishing.
The only hope is that $R\subset Q$ in all such cases. But this is not so usually. Usually the mutual position of $Q,
R$ is quite arbitrary. In the third case there are two small numbers in the denominator. This is even worse. So we are bound
for the disaster if we reduce the estimate of the operator with kernel $k$ to estimates of sums of type $I$.
But actually this is exactly the most natural way of estimating  Calder\'on-Zygmund operators. So to avoid this disaster
we have to avoid bad mutual positions of $Q,R$. This goal is attained by considering random decomposition (with respect to
random dyadic lattice)  of our functions and averaging procedure. This randomness compensates for the degeneracies of the
measure because it ``smoothens up" the degeneracies, (but however, not in a strict sense of this word). 
In another context the random dyadic lattice of course already appeared in harmonic analysis, in \cite{GJ}, for example.
Decomposition of functions to estimate   Calder\'on-Zygmund operator is not something new either, see \cite{CJS}. But the
combination of these two ideas is what allows to win over degeneracies of measures. The machinery of this is represented
below. Along with two applications (mentioned already) of this technique.

\section{Two weight estimate for the Hilbert transform. Preliminaries.} 
\label{TwoWeightPrelim}

We start now the development of two weight estimates for some Calder\'on-Zygmund operators. The technique  for degenerate
(nonhomogeneous) cases of $Tb$ theorem (see \cite{NTV1}-\cite{NTV4}) seems to work very well also for this  quite intriguing
problem from the theory of Calder\'on-Zygmund operators.

Let us recall a little bit of the history of the problem. For some time we will be mentioning only the Hilbert transform--
the common model of a Calder\'on-Zygmund operator. In 1960 Helson and Szeg\"o in \cite{HS} described the weights 
such that, say, for all $f$ smooth with compact support on the real line $\R$
\begin{equation}
\label{oneweight1}
\int_{\mathbb{R}} |Hf|^2\,wdx \leq C\, \int_{\mathbb{R}} |f|^2\,wdx\,.
\end{equation}
where the Hilbert transform $H$ is defined as follows

\begin{equation}
\label{HilbTransform}
Hf(x) := \frac1{\pi} p.v. \int_{\mathbb{R}}\frac{f(t)}{x-t}\,dt:=
\lim_{\e\rightarrow 0+}\frac1{\pi}\int_{t: |t-x| \geq \e}\frac{f(t)}{x-t}\,dt\,.
\end{equation}

Here is the description of Helson and Szeg\"o: the weight satisfies \eqref{oneweight1} if and only if
\begin{equation}
\label{HS2}
\log w = u + Hv,\,\, u, v\in L^{\infty},\,\, \|v\|_{\infty} <\frac{\pi}{2}\,.
\end{equation}

In 1971 a new description of such weights appeared. This description was due to Hunt, Muckenhaupt and Wheeden \cite{HMW},
and it was in totally different terms:
\begin{equation}
\label{HMW1}
Q_w := \sup_{I\subset \mathbb{R}} \langle w\rangle_I \langle w^{-1}\rangle_I <\infty\,.
\end{equation}
Here $I$ run over all finite intervals of the real line. It took some time to find the correct analog of this result in vector-valued situation (matrix weights), this has been done in \cite{TV1} and \cite{V} only in the 90's.
Note that so far there is no direct proof that \eqref{HMW1} implies \eqref{HS2}. 

Of course the problem with two weights attracted the attention. The problem is to describe the pairs of nonzero weights
such that
\begin{equation}
\label{twoweight1}
\int_{\mathbb{R}} |HF|^2\,vdx \leq C\, \int_{\mathbb{R}} |F|^2\,udx\,.
\end{equation}
There is a vast literature about the two weight problems. Now we mention only the works of P. Koosis \cite{PoK1},
\cite{PoK2}. 

One weight inequality became very important because of its relations with the theory of Toeplitz operators and 
with the spectral theory of stationary stochastic processes, see \cite{PKh},\cite{TV1}, \cite{TV2},\cite{V}.

Two weight inequality first attracted the attention because of its obvious relation to the one weight counterpart.
But recently it became clear that it can be very essential in perturbation theory of unitary and self-adjoint operators and in spectral theory of Jacobi matrices.
In particular, the question, when the a rank one perturbation of a unitary operator is similar to a unitary operator, is
essentially the question about the two weight estimate of the Hilbert transform, see \cite{NiT}, for example.
Subtle questions about the subspaces of the Hardy class $H^2$ invariant under the inverse shift operators
also are essentially the questions about the two weight Hilbert transform, see \cite{NV}.
And at last, see \cite{VoYu}, \cite{PVoYu} how the two weight Hilbert transform appears naturally in certain unsolved
questions concerning the orthogonal polynomials and  spectral theory of Jacobi matrices.

Let us formulate the two weight Hilbert transform problem in the form, which is more convenient to us than
\eqref{twoweight1}. Let $\mu,\nu$ be two positive measures on $\mathbb{R}$. We define the Hilbert transform $H_{\mu}$
from $L^2(\mu)$ to $L^2(\nu)$ as any bounded linear operator from $L^2(\mu)$ to $L^2(\nu)$ such that
\begin{equation}
\label{HilbTransformTwoWeight}
H_{\mu}f(x) := \frac1{\pi} \int_{\mathbb{R}}\frac{f(t)}{x-t}\,d\mu(t),\,\,\,\forall x\in \mathbb{R}\setminus
\supp (f)\,.
\end{equation}
Such an operator is not uniquely defined. But we will prove the main result for all such operators. Notice that
the adjoint $H_{\mu}*$ is just $-H_{\nu}$, it is also just a Hilbert transform in our sense (up to a minus sign).

Let us change the variables in \eqref{twoweight1}: $d\mu: =\frac1u\,dx, F:= \frac{f}{u}, d\nu := vdx$. Then 
\eqref{twoweight1} transforms itself into
\begin{equation}
\label{twoweight2}
\int_{\mathbb{R}} |H_{\mu}f|^2\,d\nu \leq C\, \int_{\mathbb{R}} |f|^2\,d\mu\,.
\end{equation}

A very subtle point is that we are not interested when \eqref{twoweight2} holds with the same finite $C$ for all
$f$ in $L^2(\mu)$. We already assumed by definition that this is the case. What we are interested in
are some simple characteristics computable by means of $\mu$ and $\nu$, and such that $C$ can be estimated by these
characteristics. An example of such characteristic is
\begin{equation}
\label{Qmunu}
Q_{\mu,\nu}:=\sup_{I\subset \mathbb{R}}\langle \mu\rangle_I\langle\nu\rangle_I:= \sup_{I\subset \mathbb{R}}\frac{\mu(I)}{|I|}
\frac{\nu(I)}{|I|}\,.
\end{equation}
This is  a total analog of $Q_w$ from \cite{HMW}. In fact, in a one weight case $u=v=w$ of \eqref{twoweight1}, we have
$d\mu=\frac1w\,dx, d\nu =w dx$, and so $Q_{\mu,\nu}$ becomes $Q_w$.
We will see soon that 
\begin{equation}
\label{necessary1}
Q_{\mu,\nu}^{1/2}\leq A\|H_{\mu}\|_{L^2(\mu)\rightarrow L^2(\nu)}\,.
\end{equation}

Of course, we are interested in a sort of opposite estimate. After all,
 Hunt-Muckenhoupt-Wheeden theorem from \cite{HMW} says
that
the finiteness of $Q_w$ is equivalent to the boundedness the corresponding Hilbert transform. Moreover, recently S.
Petermichl \cite{Petm} proved that
$$
\|H_{\mu}\|_{L^2(\mu)\rightarrow L^2(\nu)} \leq A\, Q_{\mu,\nu} =A\, Q_w
$$
in a one weight case, that is when  $d\mu=\frac1w\,dx,
d\nu =w dx$. See also \cite{PetmV}, where this is proved for the Ahlfors-Beurling transform instead of the Hilbert transform.

However in a two weight case nothing like the full analog of Hunt-Muckenhoupt-Wheeden's result (not even to mention
\cite{Petm} 
or \cite{PetmV}) is possible. Strangely enough, this has been understood only recently due to the work of F. Nazarov
\cite{N}. See also \cite{NiT}, \cite{NV}. 

At any rate $Q_{\mu,\nu}$ will be an important characteristic of ``Hunt-Muckenhoupt-Wheeden" type, which
will play an important part in estimating $\|H_{\mu}\|_{L^2(\mu)\rightarrow L^2(\nu)}$. The only thing we have said is that
it alone is not sufficient. One has to look for other $\mu,\nu$-quantities.

It is important to mention that unlike the ``Hunt-Muckenhoupt-Wheeden" type characteristics of two measures (two weights),
the ``Helson-Szeg\"o" type characteristics were found long ago. This has been done in the papers of
Cotlar and Sadosky \cite{CS2}-\cite{CS3}. Paper \cite{CS1} gives another equivalence to Helson-Szeg\'o condition.
Papers \cite{CS4}-\cite{CS5} also treat the Helson-Szeg\"o type theorem in $L^p$ for the case $p\neq 2$.

\vspace{.2in}

From what we described above it becomes clear that we are after ``Hunt-Muckenhoupt-Wheeden" type characteristics of two
measures (two weights), which, together with  characteristic $Q_{\mu,\nu}$ will allow us to estimate
$\|H_{\mu}\|_{L^2(\mu)\rightarrow L^2(\nu)}$.

The difficulty is twofold. First of all, two weight problems have a huge degree of freedom with respect to rather
rigid one weight problems. This is why one quantity $Q$ is not sufficient. Secondly we are dealing with
singular operator. Singular kernels are much more difficult to deal with than positive kernels. May be for operators with
positive kernels the two weight problems are easier approachable? It has been found in the mid 80's that this is the case.

E. Sawyer was the first who fully characterized the boundedness of several important operators with positive kernels between
two weighted spaces. This concerned in particular maximal operator and Carleson imbedding theorem. The reader is referred to
\cite{Saw1}, \cite{Saw2}, \cite{KerSaw} and also to \cite{NTV-2w}, where Sawyer's results got Bellman function explanation.
Sawyer's conditions were simple and beautiful, they were in a sense of ``Hunt-Muckenhoupt-Wheeden" type. But actually their
meaning was very transparent: 
$$
\text{a fairly general operator with positive kernel is bounded between two weighted} \,\,L^2\,\, 
$$
$$
\text{spaces if and only if it is
uniformly bounded on a system of simple test functions}
$$
$$
\text{and the same holds for its adjoint}\,.
$$
It is usually enough to take the characteristic functions of the intervals (cubes) as the family of test functions.

This was a remarkable discovery. Actually almost at the same time a series of works of G. David and J.-L. Journ\'e
appeared, devoted to the so-called $T1$ theorems. Here the main object was singular operators (kernel changes the sign), more 
precisely, Calder\'on-Zygmund operators. The answer (these $T1$ theorems) was in the same spirit: check $T$ and $T^*$ on
characteristic functions of intervals. But unlike the case considered by Sawyer, these problems of David and Journ\'e
were {\bf one weight problem}. In the following sense: given the operator with Calder\'on-Zygmund kernel $k$, bounded in
$L^2(\mu)$, one looks for characteristics, which allow to estimate the norm of this operator. The phrase ``given the operator
$T$ with Calder\'on-Zygmund kernel $k$" means that we are given a Calder\'on-Zygmund kernel $K$ and positive measure $\mu$ (say in $\R$) so that
\begin{equation}
\label{CZwithKernelk}
T_{\mu}f(x) := \int_{\R}k(x,t) f(t)\,d\mu(t)
\,\,\,\forall x\in \R\setminus
\supp (f)\,.
\end{equation}
There are many such operators of course. But David-Journ\'e were looking for characteristics which give the bound on the
norms of all such operators, meaning the norms from $L^2(\mu)$ to the same $L^2(\mu)$. 
This is why we call such problems one weight problems, they concern the estimate of $\|T_{\mu}:L^2(\mu)\rightarrow
L^2(\mu)\|$. Notice that here one weight problem means something quite different than in Hunt-Muckenhoupt-Wheeden theorem.
There one deals with $\|H_{\mu}:L^2(\mu)\rightarrow
L^2(\nu)\|$, but for a very special case: $\nu =wdx, \mu =\frac1{w}dx$.

The last important remark is that the theory of David-Journ\'e (usually united under the name ``$T1$ theorems") originally
concerned only one measure $\mu$, namely, Lebesgue measure in $\R$: $d\mu= dx$.
It was noticed that for doubling measures one can construct a series of $T1$ theorems. This has been done in a paper by
M. Christ \cite{Ch}.
The doubling property seemed to be a cornerstone of David-Journ\'e-Christ theory of Calder\'on-Zygmund operators.
However, a strong need to get rid of this cornerstone appeared from the attempt to solve Vitushkin's problems. See the reviews \cite{Ver}, \cite{Dav}, \cite{Matt}, \cite{VolLip}.

Summarizing all this: one weight problem (in both senses indicated above) are difficult, but basically solved for both
Calder\'on-Zygmund operators and for the operators with positive kernels.

Two weight problems are considerably more difficult, but basically solved for the wide class of operators with positive
kernels.

\vspace{.2in}

We consider the worst of both worlds. Our operators will be singular (we consider just some  model, for example, the Hilbert
transform or the Martingale transform) and instead of one weight problem we consider two weight problem.
This is why we need all the tricks from \cite{NTV3}--\cite{NTV5} dealing with nonhomogeneous $T1$ and $Tb$ theorems.

\vspace{.2in}

Here is our main result concerning two weight Hilbert transform.
It uses almost fully the box of tools we applied in previous papers \cite{NTV1}--\cite{NTV5} to construct a nonhomogeneous version of Calder\'on-Zygmund theory
(criteria for the boundedness of a $CZ$ operator $T_{\mu}:L^2(\mu)\rightarrow L^2(\mu)$ for nondoubling $\mu$).
The huge drawback of what has been done in \cite{NTV7} is that we were obliged to impose the doubling conditions on $\mu, \nu$ if we want
to prove a simple Hunt-Muckenhoupt-Wheeden (actually Sawyer) type result on boundedness of $H_{\mu}:L^2(\mu)\rightarrow
L^2(\nu)$. This unwelcome but returning doubling assumption is probably not needed: the result should be true in general.
But the huge difficulty of two weight estimate for singular operators forced us to impose this assumption. This is especially
strange because we use ``nonhomogeneous" technique, which is supposed to smoothen up all degeneracies of the measures. And
it does. But  so far only for one weight problems. (The recent paper \cite{LSUT} of 2010 shows this for two weight situation completely.)

In the present paper we do not impose any doubling condition on measures. But instead of getting the criteria (the necessary and sufficient condition) of boundedness of two weight Hilbert transform
we get the criteria of the two weight  boundedness for the family of operators, one operator in this family is indeed our two weight Hilbert transform. But the family consists of three operators.
Let us write the other two of them. They are standard maximal operators.
$$
M_{\mu}f (x) := \sup_{I: x\in I}\frac1{|I|}\int_I|f|\,d\mu,\,\,M_{\nu}g (x) := \sup_{I: x\in I}\frac1{|I|}\int_I|g|\,d\nu\,.
$$
By the works of E. Sawyer \cite{Saw1}, \cite{Saw2} it is known when $M_{\mu}$ is a bounded operator from $L^2(\mu)$ to
$L^2(\nu)$. This happens if and only if the uniform bound on test functions holds:
\begin{equation}
\label{Saw1}
\|M_{\mu}\chi_I\|_{\nu}^2 \leq C_M\,\mu(I),\,\,\, \forall \,\,\text{interval}\,\, I\,.
\end{equation}
The symmetric condition (with exchanging $\mu$ and $\nu$) is necessary and sufficient for the boundedness of $M_{\nu}$:
\begin{equation}
\label{Saw2}
\|M_{\nu}\chi_I\|_{\mu}^2 \leq C_M\,\nu(I),\,\,\, \forall \,\,\text{interval}\,\, I\,.
\end{equation}
Of course, the following two conditions are both necessary for the operator $H_{\mu}$ to be bounded from $L^2(\mu)$ to
$L^2(\nu)$.
\begin{equation}
\label{Cchi10}
\|H_{\mu}\chi_I\|_{L^2(\nu)}^2 \leq C_{\chi} \nu(I),\,\,\,\forall I \subset \mathbb{R}\,.
\end{equation}
\begin{equation}
\label{Cchi20}
\|H_{\nu}\chi_I\|_{L^2(\mu)}^2 \leq C_{\chi} \mu(I),\,\,\,\forall I \subset \mathbb{R}\,.
\end{equation}
They are the analogs of these Sawyer's conditions, but applied to a singular operator.

One more condition will be important to us. (It is necessary too!)
Let us recall that we introduced $Q_{\mu,\nu}$ in \eqref{Qmunu}. Its finiteness is necessary for the boundedness of the
corresponding two weight Hilbert transform. We will see this soon.
But actually there is a slightly larger quantity, more convenient for us. Its finiteness is necessary for the boundedness of the
corresponding two weight Hilbert transform as well.  Let us introduce it. Recall that 
Poisson extension of measure supported by $\mathbb{R}$ is given by the formula
$$
P_{\mu}(z) := \frac1{\pi} \int_{\mathbb{R}} \frac{\Im z}{(\Re z-t)^2 + (\Im z)^2}\,dt\,.
$$
Put
\begin{equation}
\label{PoissonQmunu}
PQ_{\mu,\nu}:=\sup_{z\in \C_+}P_{\mu}(z)P_{\nu}(z)\,.
\end{equation}
It is easy to see that there exists an absolute constant $A$ such that for any pair of positive measures
\begin{equation}
\label{QPQ}
Q_{\mu,\nu} \leq A\, PQ_{\mu,\nu}\,.
\end{equation}

\begin{thm}
\label{MainTwoWeight}
Let $\mu,\nu$ be arbitrary positive measures. Let $H_{\mu}, H_{\nu}$ be bounded on characteristic
functions, namely
\begin{equation}
\label{Cchi1}
\|H_{\mu}\chi_I\|_{L^2(\nu)} \leq C_{\chi} \nu(I),\,\,\,\forall I \subset \mathbb{R}\,.
\end{equation}
\begin{equation}
\label{Cchi2}
\|H_{\nu}\chi_I\|_{L^2(\mu)} \leq C_{\chi} \mu(I),\,\,\,\forall I \subset \mathbb{R}\,.
\end{equation}
Let also
\begin{equation}
\label{Poissonmunu}
PQ_{\mu,\nu} = \sup_{z\in \C_+}P_{\mu}(z) P_{\nu}(z) \leq C_p\,.
\end{equation}
And finally let $M_{\mu}$, $M_{\nu}$ be also bounded on characteristic functions  as written in \eqref{Saw1}, \eqref{Saw2}.
Then  the family $H_{\mu}, M_{\mu}, M_{\nu}$ consists of operators bounded by constant $C<\infty$, 
which depends only on
$C_{\chi}, C_p$ and constants in \eqref{Saw1}, \eqref{Saw2}.
\end{thm}

We can call it ``Sawyer's theorem for the family consisting of Hilbert transform and Maximal operators". Or it can be viewed as two weight version of David-Journ\'e's $T1$ theorem for the for the family consisting of Hilbert transform and Maximal operators. 

 In fact, the theorem says that the family
of three operators $H_{\mu}, M_{\mu}, M_{\nu}$ (or if you wish of four operators $H_{\mu}, M_{\mu}, H_{\nu}, M_{\nu}$ is bounded if and only if the Sawyer's conditions of testing on characteristic functions are satisfied uniformly for the operators in the family. In this respect it reminds the main result of \cite{NTV-2w}, where the family of operators was infinite (all Martingale Transforms). This is exactly what David-Journ\'e's $T1$ theorem says for Lebesgue measure, Christ's $T1$ theorem says that for an arbitrary doubling measure (but one measure, not two), and nonhomogeneous $T1$ theorem from \cite{NTV1} says the same for an arbitrary measure (but again one measure, not two).

In view of Sawyer's theorem (see \cite{Saw1}, \cite{Saw2}), we can see that this result is equivalent to the following one.
\begin{thm}
\label{MainTwoWeight1}
Let $\mu,\nu$ be arbitrary positive measures. 
Assume that $M_{\mu}$ is bounded from $L^2(\mu)$ to $L^2(\nu)$ and that that $M_{\nu}$ is bounded from $L^2(\nu)$ to $L^2(\mu)$.
Then $H_{\mu}$is bounded from $L^2(\mu)$ to $L^2(\nu)$ if and only if it is be bounded on characteristic
functions, namely
\begin{equation}
\label{Cchi1}
\|H_{\mu}\chi_I\|_{L^2(\nu)} \leq C_{\chi} \nu(I),\,\,\,\forall I \subset \mathbb{R}\,,
\end{equation}
\begin{equation}
\label{Cchi2}
\|H_{\nu}\chi_I\|_{L^2(\mu)} \leq C_{\chi} \mu(I),\,\,\,\forall I \subset \mathbb{R}\,,
\end{equation}
and also
\begin{equation}
\label{Poissonmunu}
PQ_{\mu,\nu} = \sup_{z\in \C_+}P_{\mu}(z) P_{\nu}(z) \leq C_p\,
\end{equation}
is satisfied.
\end{thm}

We believe that the assumptions of the boundedness of Maximal operators is superflous in Theorem \ref{MainTwoWeight1}.

\vspace{.3in}

\section{Necessity in the Main Theorem}
\label{necessityTW}

Assumptions \eqref{Cchi1}, \eqref{Cchi2} are obviously necessary. As to \eqref{Poissonmunu} it is necessary as well.
In fact, let us consider (just for the sake of convenience) our measures $\mu,\nu$ on the unit circle $\mathbb{T}$ (instead
of being on the line). 

As in \eqref{HilbTransformTwoWeight}, we define the two weight Hilbert transform on the circle as follows.
Let $\mu,\nu$ be two positive and {\it finite} measures on $\mathbb{T}$. We define the Hilbert transform $H_{\mu}$
from $L^2(\mu)$ to $L^2(\nu)$ as {\it any bounded linear operator} from $L^2(\mu)$ to $L^2(\nu)$ such that
\begin{equation}
\label{HilbTransformonT}
H_{\mu}f(x) := \frac1{2\pi}  \int_{\mathbb{T}}\frac{f(\zeta)}{1-\bar{\zeta}z}\,d\mu(\zeta)
,\,\,\forall x\in \mathbb{T}\setminus
\supp (f)\,.
\end{equation}
We recall that the Poisson integral of the measure on $\mathbb{T}$ is given by
\begin{equation}
\label{PoissononT}
P_{\mu}(a) := \frac1{2\pi}\int_{\mathbb{T}} \frac{1-|a|^2}{|1-\bar{a} z|^2}\,d\mu(z),\,\, a\in \mathbb{D}\,.
\end{equation}

In what follows we always consider only the measures without atoms. Here is the explanation. We want to get the necessity of 
\eqref{Poissonmunu}. Suppose $\mu, \nu$ are both delta measures at the same point.  But we adopted such a definition of $H_{\mu}$,
which allows for its non-uniqueness. Two $H_{\mu}$ may differ by the  bounded operator from $L^2(\mu)$ to $L^2(\nu)$ that preserves
the support of a function, that is by the operator of multiplication. In particular, in the case when $\mu=\nu=\delta_{1}$, we
can see that identity operator is also $H_{\mu}$. But $PQ_{\mu,\nu}=\infty$ obviously. 

\vspace{.1in}

For a point $a\in \mathbb{D}$ put $b_a(z): =\frac{z-a}{1-\bar{a}z}$. This is a Blaschke factor, it is a unimodular function
on the circle. So the operator $M_{b_a}$ of multiplication on $b_a$ is an isometry in any $L^2(\sigma), \supp\sigma\subset
\mathbb{T}$. Given a bounded operator $H_{\mu}: L^2(\mu)\rightarrow L^2(\nu)$, consider a new operator given by
$$
T_{\mu,a}:= H_{\mu} - M_{\bar{b_a}} H_{\mu} M_{b_a}\,.
$$
Then \eqref{HilbTransformonT} implies that
\begin{equation}
\label{Tmu}
T_{\mu,a}f(z) = \frac1{2\pi}\int_{\mathbb{T}} \frac{1-\overline{b_a(\zeta)} b_a(z)}{1-\bar{\zeta}z} f(\zeta)\,
d\mu(\zeta),\,\,\forall z\in
\mathbb{T}\setminus \supp(f)\,.
\end{equation}

An easy computation shows
\begin{equation}
\label{ba}
\forall \zeta, z\in \mathbb{T}, \forall a\in \mathbb{D},\,\,\,\,\,\frac{1-\overline{b_a(\zeta)} b_a(z)}{1-\bar{\zeta}z}
=
\frac{1-|a|^2}{(1-a\bar{\zeta})( 1-\bar{a}z)}\,.
\end{equation}

In particular the kernel in \eqref{Tmu} is  bounded. 

%
%
%
%
%
Let us present the idea of the rest of the proof.
The norm of such an operator (as an operator from $L^2(\mu)$ to $L^2(\nu)$) should be of course just $\|k_a\|_{\mu}\|k_a\|_{\nu}$,
which is obviously (see \eqref{PoissononT})
$(P_{\mu}(a)P_{\nu}(a))^{1/2}$.
On the other hand
\begin{equation}
\label{Tmua}
\|T_{\mu,a}\| =\|H_{\mu} - M_{\bar{b_a}}H_{\mu}M_{b_a}\|\leq 2\|H_{\mu}\|\,,
\end{equation}
as multiplications on $b_a, \bar{b_a}$ are isometries in $L^2(\mu), L^2(\nu)$.

Combining with \eqref{Tmua} one gets
\begin{equation}
\label{Poissonmunuest}
(P_{\mu}(a)P_{\nu}(a))^{1/2}\leq  2\|H_{\mu}\|\,.
\end{equation}
So we would get \eqref{Poissonmunu}. 

\vspace{.1in}

The problem with the ``proof" above is that the operator $T_a$ with the kernel 
$$
\frac{1-|a|^2}{(1-a\bar{\zeta})( 1-\bar{a}z)},
$$
and the operator $T_{\mu,a}$ may be different. In fact, \eqref{Tmu} says only that
$$
(T_{\mu,a}f,g)_{\nu} = (T_{a}f,g)_{\nu}, \,\, \forall f\in L^2(\mu), g\in L^2(\nu), \,\, \supp(f)\cap\supp(g)=\emptyset\,.
$$ 
From this  alone we do not get \eqref{Poissonmunuest}, but we get only its weaker version:
\begin{equation}
\label{Poissonmunuest1}
(P_{\mu|E}(a)P_{\nu|F}(a))^{1/2}\leq  2\|H_{\mu}\|,\,\,\,\forall E,F\subset \mathbb{T}, E\cap F =\emptyset\,.
\end{equation}

We are left to explain why \eqref{Poissonmunuest1} implies \eqref{Poissonmunuest}. They are both M\"obius invariant, and so let
$a=0$. As $\mu$ has no atoms we can choose $E_1$ to be a half-circle such that $\mu(E_1) = \frac12 \mu(\mathbb{T})$. Let $E_2
=\mathbb{T}\setminus E_1$. Call $F$ such an $E_i$ that has larger $\nu$ measure. The other one is called $E$.  For example, if
$\nu(E_1)
\geq \nu(E_2)$, we have $F=E_1, E=E_2$. Then of course, $P_{\mu}(0)P_{\nu}(0)\leq  4\,P_{\mu|E}(0)P_{\nu|F}(0)$. And
\eqref{Poissonmunuest}  with constant $2$ replaced by $4$ follows from
\eqref{Poissonmunuest1}.

\vspace{.3in}

\section{Two weight Hilbert transform. The beginning of the proof of the Main Theorem}
\label{TwoWeightBeginning}

In what follows we use Nazarov-Treil-Volberg preprint \cite{NTV7}.  F. Nazarov also noticed  that what follows can be used for a wide class of Calder\'on-Zygmund operators and not only for the Hilbert Transform. But here we consider one and the only operator--the Hilbert transform. The full criterion for the two weight boundedness of ``short range" Calder\'on-Zygmund operators (for example Martingale Transforms, Dyadic Shifts, and such...) can be found in \cite{NTV6}. In that paper no assumption on measures or any other extra assumption is used.

Let $f\in L^2(\mu), g\in L^2(\nu)$ be two test functions. We can think without the loss of generality that they have the
compact support. Then let us think that their support is in $[\frac14,\frac34]$. Let $\mathcal{D}^{\mu}, \mathcal{D}^{\nu}$
be two dyadic lattices of
$\mathbb{R}$. We can think that they are both shifts of the same standard dyadic lattice $\mathcal{D}$, such that $[0,1]\in
\mathcal{D}$, and that
$\mathcal{D}^{\mu} =\mathcal{D}+\om_1,
\mathcal{D}^{\nu}=\mathcal{D}+\om_2$, where $\om_1,\om_2\in [-\frac14,\frac14]$. We have a natural probability space of
pairs of such dyadic lattices:
$$
\Om :=\{(\om_1,\om_2) \in [-\frac14,\frac14]^2\}
$$
provided with probability $\mathbb{P}$ which is equal to normalized Lebesgue measure on $[-\frac14,\frac14]^2$.
We called these two independent dyadic lattices $\mathcal{D}^{\mu}, \mathcal{D}^{\nu}$ because they will be used to
decompose $f\in L^2(\mu), g\in L^2(\nu)$ correspondingly. This will be exactly the same type of decomposition as in the
``nonhomogeneous $T1$" theorems we met \cite{NTV1}-\cite{NTV5}. We use the notion of weighted Haar functions $h_I^{\mu}, h_I^{\nu}$, and the notion  of operators $\Delta_Q^{\mu}, \Delta_Q^{\nu}$. 

\vspace{.2in}

Let us recall that $h_I^{\mu}$ denotes the Haar function (supported by the interval $I\in \mathcal{D}^{\mu}$) with respect
to measure $\mu$. In other words, it has two values (one on the left half $I_-$ of $I$, and one on the right half $I_+$ of
$I$) such that
$$
\int_I h_I^{\mu}\,d\mu =0\,,
$$
$$
\int_I(h_I^{\mu})^2\,d\mu =1\,.
$$ 
The formula is
$$
 h_I^{\mu} = \frac1{\mu(I)^{1/2}}\Bigl[\Bigl(\frac{\mu(I_-)}{\mu(I_+)}\Bigr)^{1/2}\chi_{I_-} -
\Bigl(\frac{\mu(I_+)}{\mu(I_-)}\Bigr)^{1/2}\chi_{I_+}\Bigr],\,\, I\in \mathcal{D}^{\mu}\,.
$$

The same is for $h_I^{\nu}$ with $\mathcal{D}^{\nu}$ replacing $\mathcal{D}^{\mu}$.
We introduce the familiar operators $\Delta_I^{\mu}, \Delta_I^{\nu}$. $f\in L^2(\mu), g\in L^2(\nu)$ be two test functions
as above. Then
$$
\Delta_I^{\mu}(f) := (f,h_I^{\mu})_{\mu} h_I^{\mu},\,I\in \mathcal{D}^{\mu}, \, |I| \leq 1\,.
$$
$$
\Delta_I^{\nu}(g) := (g,h_I^{\nu})_{\nu} h_I^{\nu},\,I\in \mathcal{D}^{\nu}, \, |I| \leq 1\,.
$$
Also, let $I_0^{\mu}$ denote the interval of $\mathcal{D}^{\mu}$ of length $1$ containing $\supp(f)$, the same about
$I_0^{\nu}$ changing $f$ to $g$ and $\mu$ to $\nu$.
$$
\Lambda^{\mu}(f) := (\int_{I_0^{\mu}} f\,d\mu)\,\chi_{I_0^{\mu}},\,\Lambda^{\nu}(g) := (\int_{I_0^{\nu}}
g\,d\nu)\,\chi_{I_0^{\nu}}\,.
$$

It is easy to see that functions $\Lambda^{\mu}(f), \Delta_I^{\mu}(f),  I\in \mathcal{D}^{\mu}$ are all pairwise orthogonal 
with respect to the scalar product $(\cdot,\cdot)_{\mu}$ of $L^2(\mu)$. The same is true for $\Lambda^{\nu}(f),
\Delta_I^{\nu}(f),  I\in \mathcal{D}^{\nu}$ with respect to the scalar product $(\cdot,\cdot)_{\nu}$ of $L^2(\nu)$.
Actually, it is easy to see that the family $\chi_{I_0^{\mu}}, h_I^{\mu}, I \subset I_0^{\mu}$ is dense in the set of
functions from 
$L^2(\mu)$ supported by $[\frac14,\frac34]$. The same is true if we replace $\mu$ by $\nu$. Thus,

\begin{equation}
\label{decompmu}
f = \Lambda^{\mu}(f) + \sum_{I\in\mathcal{D}^{\mu}, I\subset I_0^{\mu}}\Delta_I^{\mu}(f),\, \,\|f\|^2_{\mu} =
\|\Lambda^{\mu}(f)\|_{\mu}^2 + \sum_{I\in\mathcal{D}^{\mu}, I\subset I_0^{\mu}}\|\Delta_I^{\mu}(f)\|_{\mu}^2\,.
\end{equation}
Similarly,

\begin{equation}
\label{decompnu}
g = \Lambda^{\nu}(g) + \sum_{I\in\mathcal{D}^{\mu}, I\subset I_0^{\mu}}\Delta_I^{\nu}(g),\,\, \|g\|^2_{\nu} =
\|\Lambda^{\nu}(g)\|_{\nu}^2 + \sum_{I\in\mathcal{D}^{\nu}, I\subset I_0^{\nu}}\|\Delta_I^{\nu}(g)\|_{\nu}^2\,.
\end{equation}

These decompositions and the assumptions \eqref{Cchi1},\eqref{Cchi2} imply in a very easy fashion that we can consider only
the case 
\begin{equation}
\label{Lambdazero}
\Lambda^{\mu}(f) = 0,\, \Lambda^{\nu}(g) =0\,. 
\end{equation}
In fact, $(H_{\mu}f, g)_{\nu} = (H_{\mu}f-\Lambda^{\mu}(f), g)_{\nu} + (\int_{I_0^{\mu}} f\,d\mu)(H_{\mu}(\chi_{I_0^{\mu}}),
g)_{\nu}$, and the second term is bounded by $C(C_{\chi}) \|f\|_{\mu}\|g\|_{\nu}$ trivially by \eqref{Cchi1}. Using
\eqref{Cchi2} one can get rid of $\Lambda^{\nu}(g)$ as well. 

So we always work under the assumption \eqref{Lambdazero}.
Now, for simplicity, we think that $f,g$ are real valued. Then

$$
(H_{\mu} f, g)_{\nu} = \sum_{I\in \mathcal{D}^{\mu}, J\in \mathcal{D}^{\nu}} (f,h_I^{\mu})_{\mu} (H_{\mu} h_I^{\mu},
h_J^{\nu})_{\nu} (g,h_I^{\nu})_{\nu}\,.
$$

\subsection{Bad and good parts of $f$ and $g$}
\label{Badandgood}

We use ``good-bad" decomposition of test functions $f,g$ exactly as this has been done in \cite{NTV1}, \cite{NTV3}--\cite{NTV5}.
Consider two fixed lattices $\mathcal{D}^{\mu}, \mathcal{D}^{\nu}$ (so we fixed a point in $\Om$, see the notations above).

We call the interval $I\in \mathcal{D}^{\mu}$ bad if there exists $J\in\mathcal{D}^{\nu}$ such that
\begin{equation}
\label{14}
|J|\geq |I|,\,\,\, \dist (e(J), I) < |J|^{3/4} |I|^{1/4}\,.
\end{equation}
Here $e(J) := \pd J\cup \text{mid point of}\, J$. Similarly one defines bad intervals $J\in \mathcal{D}^{\nu}$.

\vspace{.2in}

\noindent{\bf Definition.}
We fix a large integer $r= C(C_{\chi}, C_d)$ to be chosen later, and we say that $I\in \mathcal{D}^{\mu}$ is {\it essentially
bad} if there exists $J\in\mathcal{D}^{\nu}$ satisfying \eqref{14} such that it is much longer than $I$,
namely,
$|J|
\geq 2^r|I|$.

If the interval is not essentially bad, it is called {\it good}.

Now
\begin{equation}
\label{essbad1}
f = f_{bad} + f_{good},\,\,\, f_{bad}:= \sum_{I\in\mathcal{D}^{\mu},\,I\,\text{is essentially bad}}\Delta_I^{\mu}f\,.
\end{equation}
The same type of decomposition is used for $g$:
\begin{equation}
\label{essbad2}
g = g_{bad} + g_{good},\,\,\, g_{bad}:= \sum_{J\in\mathcal{D}^{\nu},\,J\,\text{is essentially bad}}\Delta_I^{\nu}g\,.
\end{equation}

\subsection{Estimates on good functions}
\label{goodTwoWeightgood}

We refer the reader to \cite{NTV1}, \cite{NTV3}--\cite{NTV5} for the  detailed explanation 
that it is enough to estimate $|(H_{\mu} f_{good},
g_{good})_{\nu}|$, because
\begin{equation}
\label{threeterms}
(H_{\mu} f, g)_{\nu} = (H_{\mu} f_{good}, g_{good})_{\nu} + (H_{\mu} f_{bad}, g_{good})_{\nu} + (H_{\mu} f,
g_{bad})_{\nu}\,.
\end{equation}
We repeat here sketchingly the reasoning of \cite{NTV1}, \cite{NTV3}--\cite{NTV5}. In \cite{NTV1}, \cite{NTV3}--\cite{NTV5} we proved the result that the mathematical expectation of $\|f_{bad}\|_{\mu}$,
$\|g_{bad}\|_{\nu}\|$ is small if $r$ is large. 
In fact, the proof of this fact is based on the observation that
\begin{equation}
\label{badprobTwoWeight}
\mathbb{P}\{(\om_1,\om_2)\in \Om : I\,\,\text{is essentially bad}\,| \,I\in\mathcal{D}^{\mu}\} \leq \tau(r)\rightarrow 0,
\,\,\, r\rightarrow
\infty\,.
\end{equation}

So we consider the following result as  already proved.

\begin{thm}
\label{badprobfTwoWeight}
 We consider the  decomposition of $f$ to bad and good part, and take a bad part of it for
every
$\om=(\om_1,\om_2)\in\Om$. Let $\E$ denote the expectation with respect to $(\Om,\mathbb{P})$. Then
\begin{equation}
\label{probfTwoWeight}
\E(\|f_{bad}\|_{\mu}) \leq \e(r)\|f\|_{\mu},\,\,\,\text{where}\,\,\, \e(r)\rightarrow 0, \,\,\, r\rightarrow \infty\,.
\end{equation}
The same with $g$:
\begin{equation}
\label{probgTwoWeight}
\E(\|g_{bad}\|_{\nu}) \leq \e(r)\|g\|_{\nu},\,\,\,\text{where}\,\,\, \e(r)\rightarrow 0, \,\,\, r\rightarrow \infty\,.
\end{equation}
\end{thm}

\vspace{.2in}

Coming back to \eqref{threeterms} we get
$$
|(H_{\mu} f, g)_{\nu}| \leq |(H_{\mu} f_{good}, g_{good})_{\nu}| + \|H_{\mu}\| \|f_{bad}\|_{\mu}\| g_{good}\|_{\nu} +
\|H_{\mu}\|\| f\|_{\mu} \|g_{bad}\|_{\nu}\leq
$$
$$
|(H_{\mu} f_{good}, g_{good})_{\nu}| + 2C\e(r)\|f\|_{\mu}\|g\|_{\nu},
$$
where $C$ denotes 
$\|H_{\mu}\|_{L^2(\mu)\rightarrow L^2(\nu)}$ (a priori finite, see Section \ref{TwoWeightPrelim}).
Choosing $r$ to be such that $C\e(r) < \frac14$, choosing $f,g$ to make$|(H_{\mu}
f, g)_{\nu}|$ to almost attain  $C\|f\|_{\mu}\|g\|_{\nu}$, and taking the mathematical expectation,  we get
$$
\frac12 C  \|f\|_{\mu}\|g\|_{\nu} \leq \mathbb{E}|(H_{\mu} f_{good}, g_{good})_{\nu}|
$$ 
for these special $f,g$. If we manage to prove that {\it for all} $f,g$ (see the notations for $C_d, C_{\chi}, C_p$ in
Theorem \ref{MainTwoWeight})
\begin{equation}
\label{goodgoodTwoWeight}
|(H_{\mu} f_{good}, g_{good})_{\nu}| \leq C(C_d, C_{\chi}, C_p) \|f\|_{\mu}\|g\|_{\nu}\,\,\,\forall f\in L^2(\mu),
\forall g\in L^2(\nu),
\end{equation}
then we obtain
$$
\|H_{\mu}\|_{L^2(\mu)\rightarrow L^2(\nu)} = C \leq 2 C(C_d, C_{\chi}, C_p)\,,
$$
which finishes the proof of Theorem \ref{MainTwoWeight}. 

The rest is devoted to the proof of \eqref{goodgoodTwoWeight}.

\section{First reduction of the estimate on good functions \eqref{goodgoodTwoWeight}. Long range interaction}
\label{LongRangeTwoWeight1}

So let lattices $\mathcal{D}^{\mu},\mathcal{D}^{\nu}$ be fixed, and let $f,g$ be two good functions with respect to these
lattices.
Boundedness on characteristic functions declared in \eqref{Cchi1}, \eqref{Cchi2} obviously imply
\begin{equation}
\label{obv}
|(H_{\mu} \Delta_I^{\mu} f, \Delta_J^{\nu}g)_{\nu}| \leq C_{\chi}\|\Delta_I^{\mu} f\|_{\mu}\|\Delta_J^{\nu}g\|_{\nu}\,.
\end{equation}
Therefore, in the sum $(H_{\mu}f,g)_{\nu} = \sum_{I\in \mathcal{D}^{\mu}, J\in \mathcal{D}^{\nu}}(H_{\mu}\Delta_I^{\mu},
\Delta_J^{\nu}g)_{\nu}$ the ``diagonal" part can be easily estimated. Namely, (below $r$ is the number involved in the
definition of good functions in the previous section, and we always have $I\in \mathcal{D}^{\mu}, J\in \mathcal{D}^{\nu}$
without mentioning this):
\begin{equation}
\label{diagonalTwoWeight}
\sum_{ 2^{-r}|J| \leq |I| \leq 2^r |J|, \dist(I,J) \leq\max(|I|,|J|)}|(H_{\mu}
\Delta_I^{\mu} f, \Delta_J^{\nu}g)_{\nu}| \leq C(r, C_{\chi}) \|f\|_{\mu}\|g\|_{\nu}\,.
\end{equation}

Let us consider the sums 
\begin{equation}
\label{longrangesums1}
\Sigma_1:=\sum_{ 2^{-r}|J| \leq |I|\leq |J|, \dist(I,J) \geq\ |J|}|(H_{\mu}
\Delta_I^{\mu} f, \Delta_J^{\nu}g)_{\nu}| \,.
\end{equation}
\begin{equation}
\label{longrangesums2}
\Sigma_2:=\sum_{ 2^{-r}|I| \leq |J|\leq |I|,
\dist(I,J) \geq |I|}|(H_{\mu}
\Delta_I^{\mu} f, \Delta_J^{\nu}g)_{\nu}| \,.
\end{equation}

They can be estimated in a symmetric fashion. So we will only deal with the first one.

\begin{lm}
\label{longrangelmTW}
Let $|I|\leq |J|$, $\dist (I,J) \geq |J|$. Then
\begin{equation}
\label{yellow4}
|(H_{\mu}
\Delta_I^{\mu} f, \Delta_J^{\nu}g)_{\nu}|\leq A\,\frac{|I|}{(\dist(I,J)
+|I|+|J|)^2}\mu(I)^{1/2}\nu(J)^{1/2}\|\Delta_I^{\mu} f\|_{\mu}\|\Delta_J^{\nu}g\|_{\nu}\,. 
\end{equation}
\end{lm}

\begin{proof}
Let $c$ be the center of $I$. We use the fact that $\int \Delta_I^{\mu}f\,d\mu =0$ to write
$$
(H_{\mu}
\Delta_I^{\mu} f, \Delta_J^{\nu}g)_{\nu} = \int_I d\mu(t) \int_J d\nu(s) \frac1{t-s}\Delta_I^{\mu} f(t)\Delta_J^{\nu}g(s)=
\int_I d\mu(t) \int_J d\nu(s) (\frac1{t-s}-\frac1{c-s})\Delta_I^{\mu} f(t)\Delta_J^{\nu}g(s)\,.
$$
Then one can easily see that
\begin{equation}
\label{kerneldifference}
|(H_{\mu}
\Delta_I^{\mu} f, \Delta_J^{\nu}g)_{\nu}| \leq A\, \int_I\int_J\frac{|I|}{|t-s|^2}|\Delta_I^{\mu}
f(t)||\Delta_J^{\nu}g(s)|d\mu(t)d\nu(s)\,.
\end{equation}
Now we estimate the kernel $\frac{|I|}{|J|^2 +|t-s|^2}\chi_I(t)\chi_J(s)\leq A\, \frac{|I|}{(\dist(I,J)
+|I|+|J|)^2}$ using that $|I|\leq |J|$, $\dist (I,J) \geq |J|$. On the other hand
$$
\|\Delta_I^{\mu} f\|_{L^1(\mu)} \leq \mu(I)^{1/2} \|\Delta_I^{\mu} f\|_{\mu},\,\,
\|\Delta_J^{\nu} g\|_{L^1(\nu)} \leq \mu(J)^{1/2} \|\Delta_J^{\nu} g\|_{\nu}\,.
$$
And the lemma is proved.
\end{proof}

Let us notice that Lemma \ref{longrangelmTW} allows us to write the following estimate for the sum of \eqref{longrangesums1}
(as usual $I\in \mathcal{D}^{\mu}, J\in \mathcal{D}^{\nu}$):

\begin{equation}
\label{lr1}
\Sigma_1 \leq \sum_{n=0}^{\infty}2^{-n}\sum_{I,J: |I|=2^{-n}|J|}\frac{|J|}{(\dist(I,J)
+|I|+|J|)^2}\mu(I)^{1/2}\nu(J)^{1/2}\|\Delta_I^{\mu} f\|_{\mu}\|\Delta_J^{\nu}g\|_{\nu}\,.
\end{equation}
Or
\begin{equation}
\label{lr2}
\Sigma_1\leq \sum_{n=0}^{\infty}2^{-n}\sum_{k\in\mathbb{Z}}\sum_{I,J:
|I|=2^{-n+k}, |J|=2^k}\frac{2^k}{(\dist(I,J) + 2^k)^2}\mu(I)^{1/2}\nu(J)^{1/2}\|\Delta_I^{\mu}
f\|_{\mu}\|\Delta_J^{\nu}g\|_{\nu}\,.
\end{equation}
To estimate ``the $n,k$" slice
$$
\Sigma_{n,k} := \sum_{I,J:
|I|=2^{-n+k}, |J|=2^k}\frac{2^k}{(\dist(I,J) + 2^k)^2}\mu(I)^{1/2}\nu(J)^{1/2}\|\Delta_I^{\mu}
f\|_{\mu}\|\Delta_J^{\nu}g\|_{\nu}
$$ 
let us introduce the notations.
$$
\f(t) = \sum_{I\in\mathcal{D}^{\mu}, |I| =2^{-n+k}}\frac{\|\Delta_I^{\mu}
f\|_{\mu}}{\mu(I)^{1/2}}\chi_I(t),\,\,\psi(s) = \sum_{J\in\mathcal{D}^{\nu}, |I| =2^{k}}\frac{\|\Delta_J^{\nu}
g\|_{\nu}}{\nu(J)^{1/2}}\chi_J(s)\,.
$$
Also 
$$
K_y(t,s) := \frac{y}{y^2 + |t-s|^2},\,\, y > 0, \,\, t,s \in \mathbb{R}\,.
$$
Then 
\begin{equation}
\label{lr3}
\Sigma_{n,k} \leq \int_{\mathbb{R}}d\mu(t)\int_{\mathbb{R}}d\nu(s) K_{2^k}(t,s) \f(t)\psi(s)\,.
\end{equation}

\begin{lm}
\label{PoissonTW}
The integral operator $f\rightarrow \int K_y (t,s) \f(t)\,d\mu(t)$ is bounded from $L^2(\mu)$ to $L^2(\nu)$ if
$Q_{\mu,\nu}$ (recall that this quantity is equal to $\sup_{I\subset\mathbb{R}}\langle\mu\rangle_I\langle\nu\rangle_I$) is
bounded. Its norm is bounded by
$A\, Q_{\mu,\nu}^{1/2}$.
\end{lm}

Let us postpone the proof of this lemma, and let us finish the estimate of $\Sigma_1$ using it. First of all the lemma
gives the following estimate (notice that $Q_{\mu,\nu} \leq A\,C_p$).
$$
\Sigma_{n,k} \leq C(C_p) \|\f\|_{\mu}\|\psi\|_{\nu} = C(C_p) (\sum_{I\in\mathcal{D}^{\mu},\, |I| =2^{-n+k}}\|\Delta_I^{\mu}
f\|_{\mu}^2)^{1/2}(\sum_{J\in\mathcal{D}^{\nu}, \,|J| =2^{k}}\|\Delta_J^{\nu}
g\|_{\nu}^2)^{1/2}\,.
$$
By Cauchy inequality
$$
\sum_k\Sigma_{n,k}\leq\sum_k (\sum_{J\in\mathcal{D}^{\nu},\, |J| =2^{k}}\|\Delta_J^{\nu}
g\|_{\nu}^2)^{1/2}(\sum_{I\in\mathcal{D}^{\mu},\, |I| =2^{-n+k}}\|\Delta_I^{\mu}
f\|_{\mu}^2)^{1/2} \leq 
$$
$$
(\sum_{J\in\mathcal{D}^{\nu}}\|\Delta_J^{\nu}
g\|_{\nu}^2)^{1/2}(\sum_{I\in\mathcal{D}^{\mu}}\|\Delta_I^{\mu}
f\|_{\mu}^2)^{1/2} \leq \|f\|_{\mu}\|g\|_{\nu}
$$
by \eqref{decompmu}.
Then \eqref{lr2} gives $\Sigma_1 \leq \sum_{n=0}^{\infty} 2^{-n} \sum_k \Sigma_{n,k}$, and so
$$
\Sigma_1\leq C(C_p)\sum_{n=0}^{\infty} 2^{-n}\|f\|_{\mu}\|g\|_{\nu} =2C(C_p)\|f\|_{\mu}\|g\|_{\nu},
$$
and our long range interaction sum $\Sigma_1$ is finally estimated.

\vspace{.2in}

{\bf Proof of Lemma \ref{PoissonTW}}

\vspace{.2in}

Let us consider several other averaging operators. One of them is
$$
I\f(s) := \int \chi_{[-\frac12,\frac12]}(s-t)\f(t)\,d\mu(t)\,.
$$
Another is as follows: let $G$ be all intervals $\ell_k$ of the type $
[2k, 2k+2]$, $k\in\mathbb{Z}$.
Consider
$$
A_G\f(s) := \sum_{k}\chi_{\ell_k}(s) \frac1{|\ell_k|}\int_{\ell_k}\f\,d\mu\,.
$$
Consider also shifted grid $G(x)= G+x, x\in [0,2)$, and corresponding $A_{G(x)}$.

Notice that
\begin{equation}
\label{IA}
I\f(s) \leq a\, \int_0^2 A_{G(x)}\f(s)\,dx\,.
\end{equation}
In fact, consider $[0,2], \frac12dx$ as an obvious probability space of all grids $G(x)$. Then 
it is easy to see that for every $s$ the unit interval $[s-\frac12, s+\frac12]$ is (with probability at least $1/2$) 
a subinterval of one of the intervals of $G(x)$.
Then the above inequality becomes  obvious (and $a=4$).

On the other hand, the norm of operator  $A_G$ as an operator from $L^2(\mu)$ to $L^2(\nu)$ is bounded by
$2 Q_{\mu,\nu}^{1/2}$. In fact, if  $\ell_k =
[2k, 2k+2]$, then
$$
\|A_G\f\|^2_{\nu} \leq  \sum_k (\int_{\ell_k}|\f|\,d\mu)^2\nu(\ell_k) \leq \sum_k (\int_{\ell_k}|\f|^2\,d\mu)\nu(\ell_k)
\mu(\ell_k) \leq 
$$
$$
4Q_{\mu,\nu}\sum_k \int_{\ell_k}|\f|^2\,d\mu= 4Q_{\mu,\nu}\|f\|_{\mu}^2\,.
$$
The same, of course, can be said about $\|A_{G(x)}\f\|^2_{\nu}$.
Then \eqref{IA} implies that the norm of averaging operator $I$ from $L^2(\mu)$ to $L^2(\nu)$ is bounded by $A
Q_{\mu,\nu}^{1/2}$.
Let us call by $I_r$ the operator of the same type as $I$, but the convolution now will be with the normalized
characteristic function of the interval $[-r,r]$:
$$
I_r\f(s) := \frac1{2r}\int \chi_{[-r,r]}(s-t)\f(t)\,d\mu(t)\,.
$$
It is obvious that the reasoning above can be repeated without any change and we get
\begin{equation}
\label{Ir}
\|I_r\f\|^2_{\nu} \leq A\, Q_{\mu,\nu}^{1/2}\|f\|_{\mu}^2\,.
\end{equation}

To finish with the operator given by  $f\rightarrow \int K_y (t,s) \f(t)\,d\mu(t)$ as an operator from $L^2(\mu)$ to
$L^2(\nu)$, let us notice that (and this is a standard inequality for the Poisson kernel)
$$
\int K_y (t,s) |\f(t)|\,d\mu(t)\leq A\, \sum_{k=0}^{\infty}2^{-k}(I_{y\cdot 2^k}|\f|)(s)\,.
$$
Now Lemma \ref{PoissonTW} follows immediately from \eqref{Ir} and the last inequality.

\section{ The rest of the long range
interaction}
\label{LongRangeTwoWeight2}

As always all $I$'s below are in $\mathcal{D}^{\mu}$, all $J$'s below are in $\mathcal{D}^{\nu}$.
Consider now the following two sums.

\begin{equation}
\label{longrangesums3}
\sigma_1:=\sum_{  |I|< 2^{-r}|J|, I\cap J =\emptyset}|(H_{\mu}
\Delta_I^{\mu} f, \Delta_J^{\nu}g)_{\nu}| \,.
\end{equation}
\begin{equation}
\label{longrangesums4}
\sigma_2:=\sum_{  |J|< 2^{-r}|I|,
I\cap J =\emptyset}|(H_{\mu}
\Delta_I^{\mu} f, \Delta_J^{\nu}g)_{\nu}| \,.
\end{equation}

They can be estimated in a symmetric fashion. So we will only deal with the first one.

Notice that $f,g$ are good functions. These means, in particular, that $I,J$, which we meet in \eqref{longrangesums3} satisfy
\begin{equation}
\label{14again}
\dist (I, \pd J) \geq |J|^{3/4}|I|^{1/4}\,.
\end{equation}. 
This is just \eqref{14} for disjoint $I,J$ with $I$ {\it not} essentially bad (see the definition at the
beginning of Subsection \ref{Badandgood}).

\begin{lm}
\label{longrangelmTW2}
Let $I,J$ be disjoint, $|I|< 2^{-r}|J|$,  and satisfy \eqref{14again}. Then
\begin{equation}
\label{yellow4again}
|(H_{\mu}
\Delta_I^{\mu} f, \Delta_J^{\nu}g)_{\nu}|\leq A\,\frac{|I|^{1/2}|J|^{1/2}}{(\dist(I,J)
+|I|+|J|)^2}\mu(I)^{1/2}\nu(J)^{1/2}\|\Delta_I^{\mu} f\|_{\mu}\|\Delta_J^{\nu}g\|_{\nu}\,. 
\end{equation}
\end{lm}

\begin{proof}
If $\dist (I, J) \geq |J|$, this has been already proved in Lemma \ref{longrangelmTW}. So let $\dist (I, J) \leq |J|$, $I,J$
being disjoint. Repeating \eqref{kerneldifference} one gets
$$
|(H_{\mu}
\Delta_I^{\mu} f, \Delta_J^{\nu}g)_{\nu}| \leq A\, \int_I\int_J\frac{|I|}{|t-s|^2}|\Delta_I^{\mu}
f(t)||\Delta_J^{\nu}g(s)|d\mu(t)d\nu(s)\,.
$$
Now we estimate the kernel $\frac{|I|}{|t-s|^2}\chi_I(t)\chi_J(s)\leq A\, \frac{|I|}{\dist(I,\pd J)^2}$. 
Therefore,
\begin{equation}
\label{yellow8}
|(H_{\mu}
\Delta_I^{\mu} f, \Delta_J^{\nu}g)_{\nu}| \leq A\, \frac{|I|}{\dist(I,\pd J)^2}\mu(I)^{1/2}\nu(J)^{1/2}\|\Delta_I^{\mu}
f\|_{\mu}\|\Delta_J^{\nu}g\|_{\nu}\,. 
\end{equation}
We use \eqref{14again} to write
$$
\frac{|I|}{\dist(I,\pd J)^2} \leq \frac{|I|^{1/2}}{|J|^{3/2}} = \frac{|I|^{1/2}|J|^{1/2}}{|J|^{2}}\leq A\,
\frac{|I|^{1/2}|J|^{1/2}}{(\dist(I,J)
+|I|+|J|)^2},
$$
because we assumed $\dist (I, J) \leq |J|$ and $I$ is shorter than $J$. This inequality and \eqref{yellow8} finish the proof
of the lemma.
\end{proof}

Let us notice that Lemma \ref{longrangelmTW2} allows to write the following estimate for the sum $\sigma_1$ from
\eqref{longrangesums3}:

\begin{equation}
\label{lr10}
\sigma_1 \leq \sum_{n=0}^{\infty}2^{-n/2}\sum_{I,J: |I|=2^{-n}|J|}\frac{|J|}{(\dist(I,J)
+|I|+|J|)^2}\mu(I)^{1/2}\nu(J)^{1/2}\|\Delta_I^{\mu} f\|_{\mu}\|\Delta_J^{\nu}g\|_{\nu}\,.
\end{equation}
Or
\begin{equation}
\label{lr20}
\sigma_1\leq \sum_{n=0}^{\infty}2^{-n/2}\sum_{k\in\mathbb{Z}}\sum_{I,J:
|I|=2^{-n+k}, |J|=2^k}\frac{2^k}{(\dist(I,J) + 2^k)^2}\mu(I)^{1/2}\nu(J)^{1/2}\|\Delta_I^{\mu}
f\|_{\mu}\|\Delta_J^{\nu}g\|_{\nu}\,.
\end{equation}
To estimate ``the $n,k$" slice
$$
\sigma_{n,k} := \sum_{I,J:
|I|=2^{-n+k}, |J|=2^k}\frac{2^k}{(\dist(I,J) + 2^k)^2}\mu(I)^{1/2}\nu(J)^{1/2}\|\Delta_I^{\mu}
f\|_{\mu}\|\Delta_J^{\nu}g\|_{\nu}
$$ 
let us  use again the notations
$$
\f(t) = \sum_{I\in\mathcal{D}^{\mu}, |I| =2^{-n+k}}\frac{\|\Delta_I^{\mu}
f\|_{\mu}}{\mu(I)^{1/2}}\chi_I(t),\,\,\psi(s) = \sum_{J\in\mathcal{D}^{\nu}, |I| =2^{k}}\frac{\|\Delta_J^{\nu}
g\|_{\nu}}{\nu(J)^{1/2}}\chi_J(s)\,.
$$
Also 
$$
K_y(t,s) := \frac{y}{y^2 + |t-s|^2},\,\, y > 0, \,\, t,s \in \mathbb{R}\,.
$$
Then 
\begin{equation}
\label{lr30}
\sigma_{n,k} \leq \int_{\mathbb{R}}d\mu(t)\int_{\mathbb{R}}d\nu(s) K_{2^k}(t,s) \f(t)\psi(s)\,.
\end{equation}

Lemma \ref{PoissonTW} now gives as before the estimate of $\sigma_1$. First of all the lemma
gives the following estimate (notice that $Q_{\mu,\nu} \leq A\,C_p$).
$$
\sigma_{n,k} \leq C(C_p) \|\f\|_{\mu}\|\psi\|_{\nu} = C(C_p) (\sum_{I\in\mathcal{D}^{\mu},\, |I| =2^{-n+k}}\|\Delta_I^{\mu}
f\|_{\mu}^2)^{1/2}(\sum_{J\in\mathcal{D}^{\nu}, \,|J| =2^{k}}\|\Delta_J^{\nu}
g\|_{\nu}^2)^{1/2}\,.
$$
By Cauchy inequality
$$
\sum_k\sigma_{n,k}\leq\sum_k (\sum_{J\in\mathcal{D}^{\nu},\, |J| =2^{k}}\|\Delta_J^{\nu}
g\|_{\nu}^2)^{1/2}(\sum_{I\in\mathcal{D}^{\mu},\, |I| =2^{-n+k}}\|\Delta_I^{\mu}
f\|_{\mu}^2)^{1/2} \leq 
$$
$$
(\sum_{J\in\mathcal{D}^{\nu}}\|\Delta_J^{\nu}
g\|_{\nu}^2)^{1/2}(\sum_{I\in\mathcal{D}^{\mu}}\|\Delta_I^{\mu}
f\|_{\mu}^2)^{1/2} \leq \|f\|_{\mu}\|g\|_{\nu}
$$
by \eqref{decompmu}.
Then \eqref{lr20} gives $\sigma_1 \leq \sum_{n=0}^{\infty} 2^{-n/2} \sum_k \sigma_{n,k}$, and so
$$
\sigma_1\leq C(C_p)\sum_{n=0}^{\infty} 2^{-n/2}\|f\|_{\mu}\|g\|_{\nu} =A\,C(C_p)\,\|f\|_{\mu}\|g\|_{\nu},
$$
and our long range interaction sum $\sigma_1$ is finally estimated.
Symmetric estimate holds for $\sigma_2$ from \eqref{longrangesums4}.

\vspace{.2in}

\noindent{\bf Conclusion:} if $f,g$ are good, then the sum of all terms $|(H_{\mu}
\Delta_I^{\mu} f, \Delta_J^{\nu}g)_{\nu}| $ such that either $\frac{|I|}{|J|}\in [2^{-r}, 2^r]$ or $I\cap J=\emptyset$
has the correct estimate $C(C_p)\,\|f\|_{\mu}\|g\|_{\nu}$.

\section{ The short range interaction. Corona decomposition.}
\label{ShortRangeI}

As always all $I$'s below are in $\mathcal{D}^{\mu}$, all $J$'s below are in $\mathcal{D}^{\nu}$.

Let us consider the sums 
\begin{equation}
\label{srsums1}
\rho :=\sum_{  |I|< 2^{-r}|J|, I\subset J, \dist (I,e(J)) \geq |J|^{3/4}|I|^{1/4}}(
\Delta_I^{\mu} f,H_{\nu} \Delta_J^{\nu}g)_{\mu} \,.
\end{equation}
\begin{equation}
\label{srsums2}
\tau :=\sum_{  |J|< 2^{-r}|I|, J\subset I, J\in \mathcal{D}_{\nu}, J\,\text{is good}}
(H_{\mu}\Delta_I^{\mu} f, \Delta_J^{\nu}g)_{\nu} \,.
\end{equation}
They can be estimated in a symmetric fashion. So we will only deal with, say, the second one.
It is very important that unlike the sums $\Sigma_i$, $\sigma_i$, this sum does not have absolute value on 
{\it each} term.

Consider each term of $\tau$ and split it to three terms.
To do this, let $I_i$ denote the half of $I$, which contains $J$. And $I_n$ is another half.
Let $\hat{I}$ denote an arbitrary super interval of $I_i$ in the same lattice: $\hat{I} \in \mathcal{D}^{\mu}$.


We write
$$
(H_{\mu}\Delta_I^{\mu} f, \Delta_J^{\nu}g)_{\nu} = (H_{\mu}(\chi_{I_n}\Delta_I^{\mu} f), \Delta_J^{\nu}g)_{\nu} +
(H_{\mu}(\chi_{I_i}\Delta_I^{\mu} f), \Delta_J^{\nu}g)_{\nu} =
$$
$$
(H_{\mu}(\chi_{I_n}\Delta_I^{\mu} f), \Delta_J^{\nu}g)_{\nu} + \langle \Delta_I^{\mu}f\rangle_{\mu, I_i}
(H_{\mu}(\chi_{\hat{I}}),\Delta_J^{\nu}g)_{\nu} - \langle \Delta_I^{\mu}f\rangle_{\mu, I_i}
(H_{\mu}(\chi_{\hat{I}\setminus I_i}),\Delta_J^{\nu}g)_{\nu}\,.
$$
Here $\langle \Delta_I^{\mu}f\rangle_{\mu, I_i}$ is the average of $\Delta_I^{\mu}f$ with respect to $\mu$ over $I_i$,
which is the same as value of this function on $I_i$ (by construction $\Delta_I^{\mu}f$ assumes on $I$ two values, one on
$I_i$, one on $I_n$).

\vspace{.2in}

\noindent{\bf Definition.} We call them as follows: the first one is ``the neighbor-term",
the second one is ``the difficult term", the third one is ``the stopping term".

Notice that it may happen that $\hat{I}=I_i$. Then stopping term is zero. 

\subsection{The estimate of neighbor-terms}
\label{neighborterms}

We have the same estimate as in Lemma \ref{longrangelmTW2}:

\begin{equation}
\label{nt}
|(H_{\mu}(\chi_{I_n}\Delta_I^{\mu} f), \Delta_J^{\nu}g)_{\nu}|\leq A\,\frac{|I|^{1/2}|J|^{1/2}}{(\dist(I,J)
+|I|+|J|)^2}\mu(I)^{1/2}\nu(J)^{1/2}\|\chi_{I_n}\Delta_I^{\mu} f\|_{\mu}\|\Delta_J^{\nu}g\|_{\nu}\,. 
\end{equation}
Of course, $\|\chi_{I_n}\Delta_I^{\mu} f\|_{\mu}\leq \|\Delta_I^{\mu} f\|_{\mu}$. So the estimate of the sum of absolute
values of neighbor-terms is exactly the same as the estimate of $\sigma_1$ in the preceding section. 

\subsection{The estimate of stopping terms}
\label{stoppingterms}

Here the fact that we deal with the Hilbert transform will be used in a very essential way. The estimate for other
Calder\'on-Zygmund kernels will definitely require some new tricks. We need the following definition.

\vspace{.2in}

\noindent{\bf Definition.} Given an interval $I=[a,b]$ and any measure $d\sigma$ on the real line, we write
$$
P_{[a,b]}d\sigma := \frac1{\pi}\int_{\mathbb{R}}\frac{b-a}{(b-a)^2 + ((b+a)/2 -t)^2} \,d\sigma(t)\,.
$$
This is the Poisson integral at the point whose real part is the center of the interval, 
and imaginary part is the length of
the interval.

\vspace{.2in}

We want to estimate
$$
|\langle \Delta_I^{\mu}f\rangle_{\mu, I_i}|
|(H_{\mu}(\chi_{\hat{I}\setminus I}),\Delta_J^{\nu}g)_{\nu}|\,.
$$
First of all, obviously
$$
|\langle \Delta_I^{\mu}f\rangle_{\mu, I_i}| \leq \frac{\|\Delta_I^{\mu} f\|_{\mu}}{\mu(I_i)^{1/2}}\,.
$$
Secondly,
$$
|(H_{\mu}(\chi_{\hat{I}\setminus I}),\Delta_J^{\nu}g)_{\nu}| = |(\chi_{\hat{I}\setminus I},H_{\nu}\Delta_J^{\nu}g)_{\mu}|\leq
$$
$$
A\, \Bigl(\int_{\hat{I}\setminus I}d\mu(x)\frac{|J|}{\dist(x,J)^2}\Bigr) \|\Delta_J^{\nu}g\|_{L^1(\nu)}\,.
$$ 
This is the usual trick with subtraction of the kernel, it uses the fact that $\int \Delta_J^{\nu}g\,d\nu=0$.
We continue by denoting the center of $I_i$ by $c$
$$
\leq A\,\nu(J)^{1/2}\|\Delta_J^{\nu}g\|_{\nu}\int_{\hat{I}\setminus I}\frac{\dist(x,c)^2}{\dist(x,J)^2}
\frac{|J|}{\dist(x,c)^2}\,d\mu(x)\leq
$$
$$
 A\,\nu(J)^{1/2}\|\Delta_J^{\nu}g\|_{\nu}\int_{\hat{I}\setminus I}\frac{\dist(e(I),c)^2}{\dist(e(I),J)^2}
\frac{|J|}{\dist(x,c)^2}\,d\mu(x)\,,
$$
where $e(I)$ is two end and te center of $I$. The elementary inequality above uses of course the specific nature of the
Hilbert transform. We continue, using the definition above,
$$
\leq  A\,\nu(J)^{1/2}\|\Delta_J^{\nu}g\|_{\nu}\int_{\hat{I}\setminus
I}\frac{|I|^2|J|}{|I|^{3/2}|J|^{1/2}}\frac{1}{\dist(x,c)^2}\,d\mu(x)\leq
$$
$$
A\,\nu(J)^{1/2}\|\Delta_J^{\nu}g\|_{\nu}\Bigl(\frac{|J|}{|I|}\Bigr)^{1/2} P_{I_i}(\chi_{\hat{I}\setminus
I} \,d\mu)\,.
$$
Thus 
\begin{equation}
\label{yellow9}
|(H_{\mu}(\chi_{\hat{I}\setminus
I}),\Delta_J^{\nu}g)_{\nu}|\leq A\,\nu(J)^{1/2}\|\Delta_J^{\nu}g\|_{\nu}\Bigl(\frac{|J|}{|I|}\Bigr)^{1/2}
P_{I_i}(\chi_{\hat{I}\setminus I} \,d\mu)\,.
\end{equation}

We now get the estimate of the stopping term:
\begin{equation}
\label{yellow10}
|\langle \Delta_I^{\mu}f\rangle_{\mu, I_i}|
|(H_{\mu}(\chi_{\hat{I}\setminus I}),\Delta_J^{\nu}g)_{\nu}|\leq A\,
\Bigl(\frac{\nu(J)}{\mu(I_i)}\Bigr)^{1/2}\Bigl(\frac{|J|}{|I|}\Bigr)^{1/2}P_{I_i}(\chi_{\hat{I}\setminus
I_i} \,d\mu)\|\Delta_J^{\nu}g\|_{\nu}\|\Delta_I^{\mu}f\|_{\mu}\,.
\end{equation}

\subsection{Pivotal property, which might turn out to be a  necessary condition for the two weight boundedness of the Hilbert transform}
\label{pivotalproperty}
Let $I\in \mathcal{D}_{\mu}$. Let $\{I_{\al}$ be a finite family of {\it disjoint} subintervals of $I$ belonging to the same lattice. We call the following property {\it pivotal property}:

\begin{equation}
\label{PIVOTAL}
\sum_{\al} [P_{I_{\al}}(\chi_{I\setminus I_{\al}}d\mu)]^2\nu(I_{\al}) \leq P\,\mu(I)\,.
\end{equation}

Notice that we always assume $P_{\mu}(z)P_{\nu}(z)$ uniformly bounded. (This property \eqref{Poissonmunu} is necessary for the two weight boundedness of the Hilbert transform.)
In view of this, one can replace our pivotal property by an equivalent one (may be with a different constant $P$):

\begin{equation}
\label{PIVOTAL1}
\sum_{\al} [P_{I_{\al}}(\chi_{I}d\mu)]^2\nu(I_{\al}) \leq P\,\mu(I)\,.
\end{equation}

Now properties \eqref{PIVOTAL} (or equivalently) \eqref{PIVOTAL1}  are the only things we need to prove that the Hilbert transform $H_{\mu}$ is two weight bounded if and only if $P_{\mu}(z)P_{\nu}(z)$ is uniformly bounded and test conditions \eqref{Cchi1}, \eqref{Cchi2} of Sawyer's  written down in Theorem
\ref{MainTwoWeight}.

In other words properties \eqref{PIVOTAL} (or equivalently) \eqref{PIVOTAL1}  are the only things we need to prove our two weight $T1$ theorem.

We want to emphasize that actually we do not need extra assumptions on doubling (as in \cite{NTV7})
or an extra assumption on the boundedness of maximal operators $M_{\mu}, M_{nu}$ as in this paper's Theorems \ref{MainTwoWeight}, \ref{MainTwoWeight1}. 

We only need  properties \eqref{PIVOTAL} (or equivalently) \eqref{PIVOTAL1}.  Of course with a symmetric counterpart, where places of $\mu$ and $\nu$ are exchanged. They can be {\it necessary} for the boundedness of $H_{\mu}$! If so we are done completely---a two weight $T1$ theorem is obtained with no restrictions whatsoever. But we cannot either prove or disprove the necessity of \eqref{PIVOTAL} (or equivalently) \eqref{PIVOTAL1} for the boundedness $H_{\mu}: L^2(mu)\rightarrow L^2(\nu)$.

\vspace{.1in}

\noindent{\bf Remark.} What we know is that uniform boundedness of $P_{\mu}(z)P_{\nu}(z)$ alone does not imply
\eqref{PIVOTAL}. This can be understood with the use of Bellman function method and this will be discussed in the last section of this paper.

\vspace{.1in}

However, the extra condition on doubling imposed in \cite{NTV7} allowed us to deduce \eqref{PIVOTAL}
from the boundedness $H_{\mu}: L^2(mu)\rightarrow L^2(\nu)$.

Also now we will show that \eqref{PIVOTAL} follows trivially from the assumption of boundedness $M_{\mu}: L^2(mu)\rightarrow L^2(\nu)$. This is our extra assumption in Theorems \ref{MainTwoWeight}, \ref{MainTwoWeight1}.  The symmetric counterpart of \eqref{PIVOTAL}, where places of $\mu$ and $\nu$ are exchanged, follows from the assumption of boundedness $M_{\nu}: L^2(nu)\rightarrow L^2(\mu)$.

\begin{lm}
\label{MaxOp}
Let  $M_{\mu}: L^2(mu)\rightarrow L^2(\nu)$ be bounded. Then \eqref{PIVOTAL1} holds with constant $K= A\|M_{\mu}\|^2$, where $A$ is an absolute constant.
\end{lm}

\begin{proof}
It is a standard estimate of the Poisson integral via the maximal function (see, for example, \cite{G}), which gives
\begin{equation}
\label{PmuMmu}
P_{I_{\al}}(\chi_{I}\,d\mu) \leq A\, \inf_{x\in I} (M_{\mu}\chi_{I})(x)\,.
\end{equation}
Then \eqref{PmuMmu} implies
$$
 \sum_{\al}P_{I_{\al}}(\chi_{I}\,d\mu)^2 \nu(I_{\al}) \leq
$$
$$
A\int_{\cup I_{\al}}(M_{\mu}\chi_{I})(x)^2\,d\nu(x) \leq
A\int_{I}(M_{\mu}\chi_{I})(x)^2\,d\nu(x) \leq A\,\|M_{\mu}\|^2\, \mu(I)\,.
$$

\end{proof}

\subsection{The choice of stopping intervals}
\label{stoppingchoice}

Let $K$ be a large constant to be chosen later. Fix an interval $\hat{I}\in \mathcal{D}^{\mu}$.
Let us call its {\it subinterval} $I\in \mathcal{D}^{\mu}$ a {\it stopping interval} if it is the first one (by going from bigger ones to the
smaller ones by inclusion) such that 
\begin{equation}
\label{yellow11}
\Bigl[P_{I}(\chi_{\hat{I}\setminus I}\,d\mu)\Bigr]^2 \nu(I) \geq K\, \mu(I),\,\, i=1,2\,.
\end{equation}

Here is the place, where we use the pivotal properties \eqref{PIVOTAL}:

\begin{thm}
\label{yellowS}
If $\mu, \nu$ are arbitrary positive measures such that \eqref{PIVOTAL} is satisfied, then for every $\hat{I}\in
\mathcal{D}^{\mu}$
\begin{equation}
\label{yellowCarl}
\sum_{I\in \mathcal{D}^{\mu},\, I\subset \hat{I}, \,I \text{is maximal stopping}}\mu(I) \leq \frac12 \mu(\hat{I})\,,
\end{equation}
provided that the constant $K$ in the stopping criterion \eqref{yellow11} is large enough. 
\end{thm}

\begin{proof}
In fact, let $\{I_{\al}\}$ be a family of maximal stopping intervals inside $\hat{I}$ according to stopping criteria just introduced in \eqref{yellow11}. Then
$$
\mu(I_{\al})\leq \frac1K \Bigl[P_{I_{\al}}(\chi_{\hat{I}\setminus I_{\al}}\,d\mu)\Bigr]^2 \nu(I_{\al})\,.
$$
Intervals $\{I_{\al}\}$  are disjoint subintervals of $\hat{I}$, and so \eqref{PIVOTAL} is used now:
$$
\sum_{\al} \mu(I_{\al})\leq \frac1K \sum_{\al} \Bigl[P_{I_{\al}}(\chi_{\hat{I}\setminus I_{\al}}\,d\mu)\Bigr]^2 \nu(I_{\al})
\leq
$$
$$
\frac{P}{K}\mu(\hat{I})\leq \frac12\mu(\hat{I})\,,
$$
if $K>2P$.
\end{proof}
\noindent{\bf Definitions.} 1. For any dyadic interval $I$, $F(I)$ will denote its father. 

\noindent 2. The  tree distance between the dyadic intervals of the same lattice will be denoted by $t(I_1,I_2)$. Of course $t(I, F(I))=1$.

\noindent 3. Stopping intervals  of the same lattice  will also form a tree. We will call it $\mathcal{S}$. The tree distance inside 
$\mathcal{S}$ will be denoted by $r(S_1,S_2)$. 
Of course
\begin{equation}
\label{trivialrt}
r(S_1,S_2) \leq t(S_1,S_2)\,.
\end{equation}

\subsection{Stopping tree}
\label{stoppingtree}

In Section \ref{ShortRangeI} we introduced the sum, which we are left to estimate:

\begin{equation}
\label{srsums20}
\tau :=\sum_{  |J|< 2^{-r}|I|, J\subset I, J\in \mathcal{D}_{\nu}, J\,\text{is good}}
(\Delta_I^{\mu} f, \Delta_J^{\nu}g)_{\nu} \,.
\end{equation}

Each term of $\tau$ was decomposed into three terms.
We recall: let $I_i$ denote the half of $I$, which contains $J$. And $I_n$ is another half.
Let $\hat{I}$ denote an arbitrary superinterval of $I_i$ in the same lattice: $\hat{I} \in \mathcal{D}^{\mu}$.


For a given $I\in \mathcal{D}_{\mu}$, $J\subset I, J\in \mathcal{D}_{\nu}$, $J$ good, we write down the following splitting
$$
(H_{\mu}\Delta_I^{\mu} f, \Delta_J^{\nu}g)_{\nu} = (H_{\mu}(\chi_{I_n}\Delta_I^{\mu} f), \Delta_J^{\nu}g)_{\nu} +
(H_{\mu}(\chi_{I_i}\Delta_I^{\mu} f), \Delta_J^{\nu}g)_{\nu} =
$$
\begin{equation}
\label{splittingto3terms}
(H_{\mu}(\chi_{I_n}\Delta_I^{\mu} f), \Delta_J^{\nu}g)_{\nu} + \langle \Delta_I^{\mu}f\rangle_{\mu, I_i}
(H_{\mu}(\chi_{\hat{I}}),\Delta_J^{\nu}g)_{\nu} - \langle \Delta_I^{\mu}f\rangle_{\mu, I_i}
(H_{\mu}(\chi_{\hat{I}\setminus I_i}),\Delta_J^{\nu}g)_{\nu}\,.
\end{equation}
Here $\langle \Delta_I^{\mu}f\rangle_{\mu, I_i}$ is the average of $\Delta_I^{\mu}f$ with respect to $\mu$ over $I_i$,
which is the same as value of this function on $I_i$ (by construction $\Delta_I^{\mu}f$ assumes on $I$ two values, one on
$I_i$, one on $I_n$).

We called them as follows: the first one is ``the neighbor-term",
the second one is ``the difficult term", the third one is ``the stopping term".

In what follows it is convenient to think that we consider our problem on the circle $\mathbb{T}$ rather than on the line.
We want to explain how to choose $\hat{I}$ in a stopping terms above.

\vspace{.2in}

\noindent{\bf Construction of the stopping tree $\mathcal{S}$}.
We choose first $\hat{I}=\mathbb{T}$ (this is why the circle is more convenient, we have the first ``hat" interval).
The choose its maximal stopping subintervals $\{I\}$. Just use the criterion \eqref{yellow11} from Subsection
\ref{stoppingchoice}. Call each of these $I$'s by the name $\hat{S}$. In each $\hat{S}$ again find  its maximal stopping
subintervals $\{S\}$. Et cetera... . All intervals, which were thus built, we call ``stopping intervals". They have their
generation. Stopping intervals, as a rule, will be denoted by symbols with ``hats". 

To explain the choice of $\hat{I}$ in a stopping terms above we need the notations.

\vspace{.2in}

\noindent{\bf Notations.} If $\hat{S}\in \mathcal{D}^{\mu}$ is a stopping interval, and $\SSS=\{S\}, S\in \mathcal{D}^{\mu}$
is a collection of its maximal stopping subintervals (we call them stopping suns of $\hat{S}$, there stopping tree distance to $\hat{S}$ is one: $r(S,\hat{S}) =1$), we call $\mathcal{O}_{\hat{S}}$ the collection of all intervals $I$
from {\it both} lattices $\mathcal{D}^{\mu}$, $\mathcal{D}^{\nu}$, such that the top side of the square $Q_I$ lies in the set 
$\Omega_{\hat{S}} := (\bar{Q}_{\hat{S}}\setminus \cup_{S\in \SSS} \bar{ Q_S})$. In particular, $\hat{S}\in \mathcal{O}_{\hat{S}}$, but its stopping suns are not in $\mathcal{O}_{\hat{S}}$.

\vspace{.2in}

The choice of $\hat{I}$ in a stopping terms above in \eqref{splittingto3terms} is as follows: let $I, J$ be as above, namely $J\subset I, J\in \mathcal{D}_{\nu}$, $J$ good, $J\subset I_i$, where $I_i$ is a son of $I$,
we choose the first (and unique) stopping interval $\hat{S}$ such that $I_i\in \mathcal{O}_{\hat{S}}$. Then we just put $\hat{I}=
\hat{S}$.

\vspace{.1in}

\noindent{\bf Definition.} Recall that the father of an interval $I$ with respect to the tree of all dyadic intervals was called $F(I)$. If $S\in\SSS$, then its father with respect to tree $\SSS$ will be always  called from now on $\hat{S}$.

\vspace{.2in}

Let us introduce the sum of absolute values of the ``stopping terms" of the sum $\tau$ above (as always $I\in
\mathcal{D}^{\mu}, J\in \mathcal{D}^{\nu}$).

$$
t:= \sum_{  |J|< 2^{-r}|I|, J\subset I, J\in\mathcal{D}_{\nu}, J\,\text{is good}}
|\langle \Delta_I^{\mu}f\rangle_{\mu, I_i}|
|(H_{\mu}(\chi_{\hat{I}\setminus I_i}),\Delta_J^{\nu}g)_{\nu}|\,.
$$
To estimate it we can use \eqref{yellow10}.
Then (recall that $I_i$ is the half of $I$ containing $J$)
$$
t\leq A\,T,\,\,\, T:= \sum_{  |J|< 2^{-r}|I|, J\subset I,
\dist (J,e(I)) \geq |I|^{3/4}|J|^{1/4}}
\Bigl(\frac{\nu(J)}{\mu(I_i)}\Bigr)^{1/2}\Bigl(\frac{|J|}{|I|}\Bigr)^{1/2}P_{I_i}(\chi_{\hat{I}\setminus
I_i} \,d\mu)\|\Delta_J^{\nu}g\|_{\nu}\|\Delta_I^{\mu}f\|_{\mu}\,. 
$$


We will follow the steps of \cite{NTV1} (and we will use the stopping criterion \eqref{yellow11} based on constant $K$) to
prove the following theorem.

\begin{thm}
\label{shortrangeTWthm}
$$
T\leq C(K) \|f\|_{\mu}\|g\|_{\nu}\,.
$$
\end{thm}

\begin{proof}
Put
$$
r_{n,k} := \sum_{|J|< 2^{-r}|I|, J\subset I_i, |I|= 2^{k}, |J|=
2^{-n+k}}\Bigl(\frac{\nu(J)}{\mu(I_i)}\Bigr)^{1/2}P_{I_i}(\chi_{\hat{I}\setminus I_i}
\,d\mu)\|\Delta_J^{\nu}g\|_{\nu}\|\Delta_I^{\mu}f\|_{\mu}\,.
$$
Then abusing slightly the notations we denote the halves of $I$ by $I_1, I_2$.  We get
$$
r_{n,k} \leq\sum_{i=1}^2 \sum_{|I|= 2^{k}}\|\Delta_I^{\mu}f\|_{\mu}\sum_{J\subset I_i, 
\,|J|=2^{-n+k}}\Bigl(\frac{\nu(J)}{\mu(I_i)}\Bigr)^{1/2}P_{I_i}(\chi_{\hat{I}\setminus I_i}
\,d\mu)\|\Delta_J^{\nu}g\|_{\nu}\,.
$$
Consider only $I_1$. By the Cauchy inequality the estimate will be
$$
\sum_{|I|= 2^{k}}\|\Delta_I^{\mu}f\|_{\mu}(\sum_{J\subset I_1, 
\,|J|=2^{-n+k}}\Bigl(\frac{\nu(J)}{\mu(I_1)}\Bigr)[P_{I_1}(\chi_{\hat{I}\setminus I_1}
\,d\mu)]^2)^{1/2}(\sum_{J\subset I_1, 
\,|J|=2^{-n+k}}\|\Delta_J^{\nu}g\|_{\nu}^2)^{1/2}
$$
The middle term is bounded by $[P_{I_1}(\chi_{\hat{I}\setminus I_1}\,d\mu)]^2\nu(I_1)/\mu(I_1)$. By
\eqref{yellow11}
we get that the middle term is bounded by $ K$. In fact, this was our choice of $\hat{I}$,
which ensures that $I\in \mathcal{O}_{\hat{I}}$, and so \eqref{yellow11} holds.


Thus, the last expression above is bounded by (this is just the Cauchy inequality)
$$
K\sum_{|I|= 2^{k}}\|\Delta_I^{\mu}f\|_{\mu}(\sum_{J\subset I_1, 
\,|J|=2^{-n+k}}\|\Delta_J^{\nu}g\|_{\nu}^2)^{1/2}\leq K(\sum_{|I|= 2^{k}}\|\Delta_I^{\mu}f\|_{\mu}^2)^{1/2}
(\sum_{|I|= 2^{k}}\sum_{J\subset I_1, 
\,|J|=2^{-n+k}}\|\Delta_J^{\nu}g\|_{\nu}^2)^{1/2}\,.
$$
As a result we get the estimate on $r_{n,k}$:
$$
r_{n,k} \leq C(K) \,(\sum_{|I|=
2^{k}}\|\Delta_I^{\mu}f\|_{\mu}^2)^{1/2}(\sum_{|J|=2^{-n+k}}\|\Delta_J^{\nu}g\|_{\nu}^2)^{1/2}\,.
$$
Now it is obvious from the formulae for $T$ and $r_{n,k}$  that
$$
T\leq \sum_n 2^{-n/2}\sum_k r_{n,k}\,.
$$
But from the estimate above and the Cauchy inequality $\sum_k r_{n,k}\leq C(K)\,\|f\|_{\mu}\|g\|_{\nu}$.
So we get Theorem \ref{shortrangeTWthm}.

\end{proof}

\section{Difficult terms and several paraproducts}
\label{Paraproducts}

Let us recall $f,g$ are good functions and  that in the sum
\begin{equation}
\label{srsums200}
\tau :=\sum_{  |J|< 2^{-r}|I|, J\subset I, J\in \mathcal{D}_{\nu}, J\,\text{is good}}
(H_{\mu}
\Delta_I^{\mu} f, \Delta_J^{\nu}g)_{\nu} \,.
\end{equation}
we consider each term of $\tau$ and split it to three terms.
To do this, let $I_i$ denote the half of $I$, which contains $J$. And $I_n$ is another half.
Let $S$ denote the smallest superinterval of $I_i$ in the same lattice: $S \in \mathcal{D}^{\mu}$,  $S\in \SSS$ such that
\begin{equation}
\label{defines}
I_i\in \mathcal{O}_{S}\,,
\end{equation}
where the family of intervals $\mathcal{O}_{S}$ was introduced shortly after \eqref{srsums20}. (In other  words $S$ is the smallest stopping interval containing $I_i$.)
\vspace{.2in}

We wrote
$$
(H_{\mu}\Delta_I^{\mu} f, \Delta_J^{\nu}g)_{\nu} = (H_{\mu}(\chi_{I_n}\Delta_I^{\mu} f), \Delta_J^{\nu}g)_{\nu} +
(H_{\mu}(\chi_{I_i}\Delta_I^{\mu} f), \Delta_J^{\nu}g)_{\nu} =
$$
$$
(H_{\mu}(\chi_{I_n}\Delta_I^{\mu} f), \Delta_J^{\nu}g)_{\nu} + \langle \Delta_I^{\mu}f\rangle_{\mu, I_i}
(H_{\mu}(\chi_{S}),\Delta_J^{\nu}g)_{\nu} - \langle \Delta_I^{\mu}f\rangle_{\mu, I_i}
(H_{\mu}(\chi_{S\setminus I_i}),\Delta_J^{\nu}g)_{\nu}\,.
$$
Here $S$ is the smallest interval from the stopping tree $\SSS$ such that $I_i\in \mathcal{O}_{S}$.  Also h ere $\langle \Delta_I^{\mu}f\rangle_{\mu, I_i}$ is the average of $\Delta_I^{\mu}f$ with respect to $\mu$ over $I_i$,
which is the same as value of this function on $I_i$ (by construction $\Delta_I^{\mu}f$ assumes on $I$ two values, one on
$I_i$, one on $I_n$).

The sum of absolute values of the first terms and the sum of absolute values of the third terms were already
bounded by $C\|f||_{\mu}\|g\|_{\nu}$ in the preceding sections. Middle terms were called ``difficult terms", and we are
going to estimate the absolute value of the sum of all difficult terms now. This is the most difficult part of the proof.

\vspace{.2in}

Let $\{S\}_{S\in \SSS}$ denote  the family of stopping type intervals of all generations (for the convenience we think that
we are on the circle $\mathbb{T}$ and the first generation consists of the circle itself).
{\it In what follows the letter $S$ is reserved for the stopping intervals.}  Recall that $\hat{S}$ also denotes the stopping interval, the father of $S$ inside the stopping tree $\SSS$.

\vspace{.2in}

\noindent{\bf Notations.} Let $S\in\SSS$ be an arbitrary stopping interval. We denote by
$\mathbb{P}_{\mu, \mathcal{O}_S}$ the orthogonal projection in $L^2(\mu)$ onto the space generated by
$\{h_I^{\mu}\}$, $I\in \mathcal{O}_S$, $I$ is good, and we denote by
$\mathbb{P}_{\nu, \mathcal{O}_S}$ the orthogonal projection in $L^2(\nu)$ onto the space generated by
$\{h_J^{\nu}\}$, $J\in \mathcal{O}_S$, $J$ is good. (Recall that $\mathcal{O}_S$ included by definition the intervals 
in both lattices $\mathcal{D}^{\mu}$ and $\mathcal{D}^{\nu}$.)

\vspace{.2in}

We fix $I\in \mathcal{D}^{\mu}$, it defines $S\in \SSS$ (see \eqref{defines}),  we look at terms
$$
\langle \Delta_I^{\mu}f\rangle_{\mu, I_i}
(H_{\mu}(\chi_{S}),\Delta_J^{\nu}g)_{\nu}\,.
$$

We can write each of the term  $\langle \Delta_I^{\mu}f\rangle_{\mu, I_i}
(H_{\mu}(\chi_{S}),\Delta_J^{\nu}g)_{\nu}$ with fixed $S$ and $I\in \mathcal{O}_{S}, J\in
\mathcal{O}_{S}$ as 
$$
\langle \Delta_I^{\mu}\mathbb{P}_{\mu, \mathcal{O}_{S}}f\rangle_{\mu,
I_i}(H_{\mu}(\chi_{S}),\Delta_J^{\nu}\mathbb{P}_{\nu,
\mathcal{O}_{S}}g)_{\nu}\,.
$$

\vspace{.2in}

\noindent{\bf The definition of $\tau_{S}$ .}
We collect all of these terms with 
$I\in \mathcal{O}_{S}, I\in \mathcal{D}^{\mu}, J\in \mathcal{O}_{S}$, $J\in \mathcal{D}^{\nu}, |J| \leq 2^{-r}
|I|$, $J$ is good. The resulting sum is called $\tau_{S}$.
(In summation below we should remember that  $f,g$ are good: so we can sum over all pertinent pairs of $I,J$ remembering
that some of $\Delta$'s are zero anyway.)

\vspace{.2in}

We first fix good $J$, then summing over such $I$'s gives (such $I$'s should contain $J$, and they form a ``tower" of nested
intervals, from the smallest one called $\ell(J)$ to the largest one equal to $S$; notice that the summing of
quantities
$\langle
\Delta_I^{\mu}\f\rangle_{\mu, I}$ over such a ``tower" results in the average over the smallest interval minus the average
over the largest interval of the ``tower", the latter one being zero in our case)
$$
\langle \mathbb{P}_{\mu, \mathcal{O}_{S}}f\rangle_{\mu,
\ell(J)}(\Delta_J^{\nu}H_{\mu}(\chi_{S}),\mathbb{P}_{\nu,
\mathcal{O}_{S}}g)_{\nu}\,,
$$
where $\ell(J)\in \mathcal{O}_{S},\ell(J)\in \mathcal{D}^{\mu}, |\ell(J)| = 2^{r-1} |J|$.
One can argue that replacing $f$ by $ \mathbb{P}_{\mu, \mathcal{O}_{S}}f$ we make gaps in the tower as $\langle \Delta_I^{\mu}f\rangle_{\mu,I}$ got replaced by $0$ from time to time (for bad $I$'s actually). But this is not a problem as $f$ is good, and so $\langle \Delta_I^{\mu}f\rangle_{\mu,I}$ is zero anyway for bad $I$'s!

\vspace{.1in}

Summing over $J$ we get 
$$
\tau_{S}=\sum_{I\in\mathcal{D}^{\mu}, I\in  \mathcal{O}_{S}, I \,\text{is good}}\langle \mathbb{P}_{\mu,
\mathcal{O}_{S}}f\rangle_{\mu, I}(\sum_{J\in\mathcal{D}^{\nu}, J\in \mathcal{O}_{S},  |J|=
2^{-r+1}|I|, J\,\text{is good}}\Delta_J^{\nu}H_{\mu}(\chi_{S}),
\mathbb{P}_{\nu,
\mathcal{O}_{S}}g)_{\nu}\,.
$$

\subsection{First paraproduct}
\label{FirstParaproduct}

Let us introduce our first paraproduct operator
$$
\pi_{H_{\mu}\chi_{S}} \f := \sum_{I\in\mathcal{D}^{\mu}, I\in  \mathcal{O}_{S}}
\langle \f\rangle_{\mu,I}\sum_{J\in\mathcal{D}^{\nu}, J\in \mathcal{O}_{S}, J\subset I,  |J|=
2^{-r+1}|I|, J\,\text{is good}}\Delta_J^{\nu}H_{\mu}(\chi_{S})\,.
$$
Then the absolute value of the sum $\tau_{S}$ above is
\begin{equation}
\label{C1}
|(\pi_{H_{\mu}\chi_{S}}\mathbb{P}_{\mu, \mathcal{O}_{S}}f, \mathbb{P}_{\nu,
\mathcal{O}_{S}}g)_{\nu}|\leq C_1\,\|\mathbb{P}_{\mu, \mathcal{O}_{S}}f\|_{\mu} \|\mathbb{P}_{\nu,
\mathcal{O}_{S}}g\|_{\nu}\,,
\end{equation}
where $C_1$ is the norm of $\pi_{H_{\mu}\chi_{S}}$ as an operator from $L^2(\mu)$ to $L^2(\nu)$.

\begin{thm}
\label{FirstParaproduct}
The norm of operator $\pi_{H_{\mu}\chi_{S}}$ as an operator from $L^2(\mu)$ to $L^2(\nu)$ is bounded by
$C_1(K)<\infty$, where $K$ is the constant participating in the definition of stopping intervals.
\end{thm} 

\begin{proof}
Obviously
$$
\|\pi_{H_{\mu}\chi_{S}}\f\|_{\nu}^2 \leq \sum_{I\in\mathcal{D}^{\mu}, I\in  \mathcal{O}_{S}}
|\langle \f\rangle_{\mu,I}|^2\,a_I,
$$
where $\F(I) := \{ J: J\in\mathcal{D}^{\nu}, J\in \mathcal{O}_{S}, J\subset I,  |J|=
2^{-r+1}|I|, \dist(J,\pd(I))\geq |I|^{3/4}|I|^{1/4}\}$
$$
a_I := \sum_{J\in\F(I)}\|\Delta_J^{\nu}H_{\mu}(\chi_{S})\|_{\nu}^2\,.
$$
The Carleson imbedding theorem (see \cite{G}, and in this context \cite{NTV1}) says that the boundedness of the sum
$\sum_{I\in\mathcal{D}^{\mu}, I\in  \mathcal{O}_{S}}
|\langle \f\rangle_{\mu,I}|^2\,a_I$ by $C\, \|\f\|_{\mu}^2$ is equivalent to the following Carleson condition
\begin{equation}
\label{CarlC11}
\forall I\in \mathcal{D}^{\mu},\,I\in  \mathcal{O}_{S}\, \sum_{\ell\in\mathcal{D}^{\mu}, \ell\in  \mathcal{O}_{S},
\ell\subset
I}a_{\ell} \leq c\,\mu(I)
\end{equation}
Of course ($\Psi(I)  := \{ J: J\in\mathcal{D}^{\nu}, J\in \mathcal{O}_{S}, J\subset I,  |J|\leq
2^{-r+1}|I|, \dist(J,\pd(I))\geq |I|^{3/4}|I|^{1/4}\}$)
$$
\sum_{\ell\in\mathcal{D}^{\mu}, \ell\in  \mathcal{O}_{S}, \ell\subset I}a_{\ell} =\sum_{ J:
J\in\Psi(I)}\|\Delta_J^{\nu}H_{\mu}(\chi_{S})\|_{\nu}^2 =\| \sum_{ J:
J\in\Psi(I)}\Delta_J^{\nu}H_{\mu}(\chi_{S})\|_{\nu}^2 \,.
$$
By duality then
$$
\sum_{\ell\in\mathcal{D}^{\mu}, \ell\in  \mathcal{O}_{S}, \ell\subset I}a_{\ell} = \sup_{\psi\in L^2(\nu), \,\|\psi\|_{\nu}
=1}|
\sum_{ J:
J\in\Psi(I)}(H_{\mu}(\chi_{S}), \Delta_J^{\nu}\psi)_{\nu}|^2\leq 
$$
$$
\sup_{\psi\in L^2(\nu), \,\|\psi\|_{\nu} =1}|
\sum_{ J:
J\in\Psi(I)}(H_{\mu}(\chi_{S\setminus I}), \Delta_J^{\nu}\psi)_{\nu}|^2 + \|H_{\mu}(\chi_{ I})\|_{\nu}^2\,.
$$
So \eqref{Cchi1} implies
\begin{equation}
\label{CarlC12}
\sum_{\ell\in\mathcal{D}^{\mu}, \ell\in  \mathcal{O}_{S}, \ell\subset I}a_{\ell}  \leq \sup_{\psi\in L^2(\nu),
\,\|\psi\|_{\nu} =1}|
\sum_{ J:
J\in\Psi(I)}(H_{\mu}(\chi_{S\setminus I}), \Delta_J^{\nu}\psi)_{\nu}|^2 + C_{\chi}\,\mu(I)\,.
\end{equation}

Let us consider the term $(H_{\mu}(\chi_{S\setminus I}), \Delta_J^{\nu}\psi)_{\nu}$, $J\in \Psi(I)$. Exactly this
quantity was estimated in  \eqref{yellow9}. We get

$$
|(H_{\mu}(\chi_{S\setminus I}), \Delta_J^{\nu}\psi)_{\nu}|\leq A\, \nu(J)^{1/2}
\|\Delta_J^{\nu}\psi\|_{\nu}\Bigl(\frac{|J|}{|I|}\Bigr)^{1/2} P_I(\chi_{S\setminus I})\,d\mu\,.
$$
So the first term in \ref{CarlC12} is bounded by (we use the Cauchy inequality)
$$
\sum_{ J:
J\in\Psi(I)} \frac{|J|}{|I|}[P_I(\chi_{S\setminus I})\,d\mu]^2\nu(J)\leq \sum_n
2^{-n}\sum_{|J|=2^{-n}|I|,J\subset I}[P_I(\chi_{S\setminus I})\,d\mu]^2\nu(J) =
$$
$$
\sum_n
2^{-n}[P_I(\chi_{S\setminus I})\,d\mu]^2\nu(I)
$$
as $\|\psi\|_{\nu} =1$. It is time to use the fact that $I\in \mathcal{O}_{S}$, 
which means that the stopping criterion
\eqref{yellow11} is not yet achieved on $I$, in other words that
$$
[P_I(\chi_{S\setminus I})\,d\mu]^2\nu(I)\leq K\, \mu(I)\,.
$$
Combining this with \eqref{CarlC12} we get \eqref{CarlC11}:
$$
\sum_{\ell\in\mathcal{D}^{\mu}, \ell\in  \mathcal{O}_{S}, \ell\subset I}a_{\ell}  \leq (K+C_{\chi})\,\mu(I)\,.
$$
And Theorem \ref{FirstParaproduct} is proved.

\end{proof}

Let us recall that we introduced above the definition of $\tau_{S}$, for stopping interval $S$.
We finished the estimate of 
the sum of $\tau_{S}$ over all stopping $S$ (recall that the set of all, stopping intervals was called $\SSS$):

\begin{equation}
\label{inSinS}
\sum_{S\in \SSS}\tau_{S} \leq C(K,C_{\chi}) \sum_{S\in\SSS}\|\mathbb{P}_{\mu,
\mathcal{O}_{S}}f\|_{\mu}\|\mathbb{P}_{\nu,
\mathcal{O}_{S}}g\|_{\nu}\leq  C(K,C_{\chi})\|f\|_{\mu}\|g\|_{\nu},,
\end{equation}
the last inequality following from the orthogonality of $\mathbb{P}_{\mu,
\mathcal{O}_{S}}f$ for different $S$ (the same for $\mathbb{P}_{\nu,
\mathcal{O}_{S}}g$) and the Cauchy inequality.

\subsection{Two more paraproducts}
\label{Twomoreparaproducts}

In the previous subsection we have estimated a piece of the sum of the difficult terms
\begin{equation}
\label{rest}
\langle \Delta_I^{\mu}f\rangle_{\mu, I_i}
(H_{\mu}(\chi_{S}),\Delta_J^{\nu}g)_{\nu}\,,
\end{equation}
namely, we estimated the sum of such terms, when $I,J$ lie both in the same family
 $\mathcal{O}_{S}$, where $S\in \SSS$ (arbitrary stopping interval). Such a sum was called $\tau_S$, and we just proved in
\eqref{inSinS} that
$\sum_{S\in\SSS} \tau_S \leq C\|f\|_{\mu}\|g\|_{\nu}$.

What is left is to estimate the sum of abovementioned terms when $J\in \mathcal{O}_{S}$ and $I$ belongs to another
$\mathcal{O}_{S_1}$, where $S, S_1$ are both stopping intervals. As $I$ is larger than $J$, we have to consider the
pairs of stopping intervals, where $S$ is strictly inside $S_1$ ($S_1$ is one or more  generations higher in a stopping tree $\SSS$ than $S$).

Let us recall that $F(I)$ denote the father of $I$ inside the standard dyadic tree.
Let us fix $J$. Let $...\subset S_3\subset S_2 \subset S_1\subset ...$ be a (finite) sequence of stopping intervals
of successive generations containing $J$.  So $S_{i-1}$ is a father of $S_i$ in the stopping tree $\SSS$. So it is notv true that $S_{i-1}=F(S_i)$ in general!

The sequence for $I$'s, over which we have to sum up,  will be one term shorter
(the smallest one should be discarded). This is  because we sum up all the terms, where $J$ and $I$ are in different families
$\mathcal{O}_{S_i}, \mathcal{O}_{S_{i-1}}$, and $S_i$ is inside $S_{i-1}$. Notice also that $\langle \Delta_I^{\mu}f\rangle_{\mu, I_i}$
is the difference between two averages of $f$ with respect to $\mu$, one over $I_i$ and one over its father $I$. It is easy 
to some up successive differences and summing all above mentioned terms with fixed $J$ we get

$$
...+ (\langle f\rangle_{\mu, S_2} - \langle f\rangle_{\mu, F(S_2)}) (H_{\mu}\chi_{S_2}, \Delta_J^{\nu} g)_{\nu} +
(\langle f\rangle_{\mu, F(S_3)} - \langle f\rangle_{\mu, F(F(S_3))}) (H_{\mu}\chi_{S_2}, \Delta_J^{\nu} g)_{\nu} + 
$$
$$
(\langle f\rangle_{\mu, S_3} - \langle f\rangle_{\mu, F(S_3)}) (H_{\mu}\chi_{S_3}, \Delta_J^{\nu} g)_{\nu} +...
$$


Regrouping, we get
$$
...+ \langle f\rangle_{\mu, F(S_3)}(H_{\mu}\chi_{S_2\setminus S_3}, \Delta_J^{\nu} g)_{\nu} +...
$$

We have to take into considerations also the terms with the smallest $S_m$ for a given $J$, for which there will be no pair.
Subsequently, the sum of abovementioned terms in \eqref{rest}, when $J\in \mathcal{O}_{S}$ and $I$ belongs to another
$\mathcal{O}_{\tilde{S}}$, where $S, \tilde{S}$ are both stopping intervals, $S$ is strictly smaller than $\tilde{S}$, can be
written in the following form. (We denote by $\hat{S}$ the stopping interval containing the stopping $S$ and of the previous
generation (the stopping father of $S$).

\vspace{.1in}

\noindent{\bf Warning.} In all sums we meet below $J$ is {\it always} good. We have to add everywhere ``$J$ is good". For the sake of brevity we do not do that, but we ask the reader to keep this in mind.

\vspace{.1in}

This is what we are left to estimate.
\begin{equation}
\label{restsum}
\rho := \sum_{s\in\SSS} \langle f\rangle_{\mu, F(S)} (H_{\mu} \chi_{\hat{S} \setminus S}, \sum_{J\in Q_S} \Delta_J^{\nu}g)_{\nu}
+ \sum_{s\in\SSS} \langle f\rangle_{\mu, S} (H_{\mu} \chi_{S}, \sum_{J\in \mathcal{O}_S} \Delta_J^{\nu}g)_{\nu}\,.
\end{equation}
We used the notations $Q_S = \cup_{s\in\SSS, s\subset S} \mathcal{O}_s$. This
means that the family $Q_S$ consists of intervals $\ell$ from our both lattices $\mathcal{D}^{\mu}, \mathcal{D}^{\nu}$ such
that the top of the square $Q_{\ell}$ belongs to the square $Q_S$ (the slight abuse of notations, the square, and the
corresponding family of the intervals are denoted by the same letter). Recall that $J$ is always good in the above sums.
Then
we can introduce two projections $\mathbb{P}_{\nu, Q_S}$, $\mathbb{P}_{\nu, \mathcal{O}_S}$. Actually the second one was
already introduced. But anyway, we denote by
$\mathbb{P}_{\nu, \mathcal{O}_S}$ the orthogonal projection in $L^2(\nu)$ onto the space generated by
$\{h_J^{\nu}\}$, $J\in \mathcal{O}_S$, $J$ is good. And we denote by $\mathbb{P}_{\nu, Q_S}$ the orthogonal projection in $L^2(\nu)$ onto
the space generated by
$\{h_J^{\nu}\}$, $J\in Q_S$, $J$ is good. (Recall that $Q_S, \mathcal{O}_S$ included by definition the intervals 
from both lattices $\mathcal{D}^{\mu}$ and $\mathcal{D}^{\nu}$.) 
Now we can write $\rho$ as follows
$$
\rho = \sum_{s\in\SSS} \langle f\rangle_{\mu, F(S)} (H_{\mu} \chi_{\hat{S} \setminus S},\mathbb{P}_{\nu, Q_S}g )_{\nu} +
\sum_{s\in\SSS} \langle f\rangle_{\mu, S} (H_{\mu} \chi_{S}, \mathbb{P}_{\nu, \mathcal{O}_S}g)_{\nu} =: \rho_1 +\rho_2\,.
$$

\vspace{.2in}

We introduce now  two paraproducts:

$$
\pi^{\mathcal{O}} f: = \sum_{s\in\SSS} \langle f\rangle_{\mu, S} \mathbb{P}_{\nu, \mathcal{O}_S}(H_{\mu} \chi_{S})\,,
$$
$$
\pi^{Q}f : = \sum_{s\in\SSS} \langle f\rangle_{\mu, F(S)} \mathbb{P}_{\nu, Q_S}(H_{\mu} \chi_{\hat{S}\setminus S})\,.
$$

Then $\rho_1 = (\pi^{\mathcal{O}} , g)_{\nu}, \rho_2 = (\pi^{Q} , g)_{\nu}$. So to finish the proof of our main Theorem
\ref{MainTwoWeight} it is enough to prove the boundedness of these paraproducts as operators from $L^2(\mu)$ to
$L^2(\nu)$.

To prove the boundedness of the first paraproduct let us use Theorem \ref{yellowS}. Consider the sequence
$$
\{b_S\}_{S\in\SSS},\,\, b_S := \|\mathbb{P}_{\nu, \mathcal{O}_S}(H_{\mu} \chi_{S})\|_{\nu}^2 \,.
$$
It is a Carleson sequence:
\begin{equation}
\label{CarlaS}
\forall I\in \mathcal{D}^{\mu}\,\,\sum_{S\subset I, S\in \SSS} b_S \leq C\, \mu(I)\,.
\end{equation}
In fact, $ b_S \leq  \|H_{\mu} \chi_{S}\|_{\nu}^2\leq C_{\chi}\, \mu(S)$ by \eqref{Cchi1}. Now \eqref{CarlaS} becomes clear
by Theorem \ref{yellowS}.

Notice that $\mathbb{P}_{\nu, \mathcal{O}_S}$ are mutually orthogonal projections in $L^2(\nu)$ for different $S$. 
This is just because the families $\mathcal{O}_S$ are pairwise disjoint for different $S\in \SSS$.
This is exactly what helped us to cope with $\pi^{\mathcal{O}} f$ so easily, we just used
$$
\|\pi^{\mathcal{O}} f\|_{\nu}^2= \|\sum_{s\in\SSS} \langle f\rangle_{\mu, S} \mathbb{P}_{\nu, \mathcal{O}_S}(H_{\mu} \chi_{S})\|_{\nu}^2= \sum_{S\in\SSS}| \langle f\rangle_{\mu, S}|^2\| \mathbb{P}_{\nu, \mathcal{O}_S}(H_{\mu} \chi_{S})\|_{\nu}^2=  \sum_{S\in\SSS}| \langle f\rangle_{\mu, S}|^2 a_S\,.
$$
This is where the orthogonality has been used. And we applied then the Carleson property of $\{b_S\}_{S\in\SSS}$.
We
already saw this type of paraproducts with the property of orthogonality (see
\cite{G},
\cite{NTV1}, and especially Theorem
\ref{FirstParaproduct} above). And we know  that Carleson condition
\eqref{CarlaS} is sufficient for the paraproduct operator $\pi^{\mathcal{O}}$ to be bounded.

\vspace{.2in}

\noindent{\bf The second paraproduct $\pi^{Q}$.}

\vspace{.2in}
 
The main problem is that $\mathbb{P}_{\nu, Q_S}$ {\it are not mutually orthogonal} projections in $L^2(\nu)$.

So $\|\pi^{Q} f\|_{\nu}^2$ has the diagonal part but also the out of diagonal par:
$$
\|\pi^{Q} f\|_{\nu}^2 \leq  DP + ODP\,,
$$
where
$$
DP := \sum_{S\in \SSS} |\langle f\rangle_{\mu,F(S)}|^2 \|\mathbb{P}_{\nu, Q_S} H_{\mu}(\chi_{\hat{S}\setminus S})\|_{\nu}^2\,,
$$
$$
ODP := \sum_{S, S'\in \SSS, S'\subset S, S'\neq S}|\langle f\rangle_{\mu,F(S')}||\langle f\rangle_{\mu,F(S)}| |(\mathbb{P}_{\nu,
Q_S} H_{\mu}(\chi_{\hat{S}\setminus S}), \mathbb{P}_{\nu, Q_{S'}} H_{\mu}(\chi_{\hat{S'}\setminus S'})_{\nu}| =
$$
$$
\sum_{S, S'\in \SSS, S'\subset S, S'\neq S}|\langle f\rangle_{\mu,F(S')}||\langle f\rangle_{\mu,F(S)}| |(\mathbb{P}_{\nu,
Q_{S'}} H_{\mu}(\chi_{\hat{S}\setminus S}), \mathbb{P}_{\nu, Q_{S'}} H_{\mu}(\chi_{\hat{S'}\setminus S'})_{\nu}|\,.
$$
We start with $ODP$. Recall that $r=r(S', S)$ is the generation gap between $S'$ and $S$, $S'\subset S$ in the stopping tree $\SSS$.
$$
ODP \leq \sum_{S, S'\in \SSS, S'\subset S, S'\neq S}|\langle f\rangle_{\mu,F(S)}|^2 \|\mathbb{P}_{\nu, Q_{S'}}
H_{\mu}(\chi_{\hat{S}\setminus S})\|_{\nu}^2\cdot (1+\e)^{r(S',S)} + 
$$
$$
\sum_{S, S'\in \SSS, S'\subset S, S'\neq S}
|\langle f\rangle_{\mu,F(S')}|^2 \|\mathbb{P}_{\nu, Q_{S'}}
H_{\mu}(\chi_{\hat{S'}\setminus S'})\|_{\nu}^2\cdot (1+\e)^{-r(S',S)}\leq
$$
$$
\sum_{S\in\SSS} |\langle f\rangle_{\mu,F(S)}|^2\sum_{j=1}^{\infty} (1+\e)^j \sum_{S'\in \SSS, S'\subset S, r(S',S)
=j}\|\mathbb{P}_{\nu, Q_{S'}} H_{\mu}(\chi_{\hat{S}\setminus S})\|_{\nu}^2 + C(\e) \sum_{S\in\SSS} |\langle
f\rangle_{\mu,F(S)}|^2 \|\mathbb{P}_{\nu, Q_{S}} H_{\mu}(\chi_{\hat{S}\setminus S})\|_{\nu}^2\,.
$$

Now we need to estimate these  sums
$$
F_j := \sum_{S\in\SSS} |\langle f\rangle_{\mu,F(S)}|^2\sum_{S'\in \SSS, S'\subset S, r(S',S)
=j}\|\mathbb{P}_{\nu, Q_{S'}} H_{\mu}(\chi_{\hat{S}\setminus S})\|_{\nu}^2, \,\, j=1,2,3,...\,,
$$
$$
F_0 := \sum_{S\in\SSS} |\langle
f\rangle_{\mu,F(S)}|^2 \|\mathbb{P}_{\nu, Q_{S}} H_{\mu}(\chi_{\hat{S}\setminus S})\|_{\nu}^2\,.
$$
By the way, $F_0=DP$. 

All such sums have the form of Carleson imbedding theorems. So we need to check countable number of Carleson 
conditions now.

\vspace{.2in}

\noindent{\bf Carleson condition for $F_j$}.
We introduce the sequence
$$
a_S :=  \|\mathbb{P}_{\nu, Q_{S}} H_{\mu}(\chi_{\hat{S}\setminus S})\|_{\nu}^2,\,\,S, \hat{S}\in\SSS, r(S, \hat{S})=1\,.
$$
And also
$$
a^j_S :=\sum_{S'\in\SSS, S'\subset S, r(S', s) =j} \|\mathbb{P}_{\nu, Q_{S'}} H_{\mu}(\chi_{\hat{S}\setminus S})\|_{\nu}^2,\,r(S, \hat{S})=1,\,\,\, j=1,2,3,...\,.
$$
We will need the following Lemma.

\begin{lm}
\label{projectionviaPoisson}
Let $A'\subseteq A\subset B$ be intervals of $\mathcal{D}_{\mu}$. Let the tree distance between $A'$ and $A$ with respect to the tree $\mathcal{D}_{\mu}$ satisfy $t(S',S) \geq j,\, j= 0, 1, 2,..$. Then
$$
\|\bP_{\nu, A'} (H_{\mu}\chi_{B\setminus A})\|_{\nu}^2 \leq C\, 2^{-j} \nu(A')(P_A\chi_{B\setminus A} d\mu)^2\,.
$$
\end{lm}

\begin{proof}
Let $\|\psi\|_{\nu}=1$. Let us consider the term $(H_{\mu}(\chi_{B\setminus A }), \Delta_J^{\nu}\psi)_{\nu}$, $J\in
Q_{A'}$. Exactly this quantity was estimated in  \eqref{yellow9}. We get

$$
|(H_{\mu}(\chi_{B\setminus A}), \Delta_J^{\nu}\psi)_{\nu}|\leq C, \nu(J)^{1/2}
\|\Delta_J^{\nu}\psi\|_{\nu}\Bigl(\frac{|J|}{|A|}\Bigr)^{1/2} P_A(\chi_{B\setminus A}\,d\mu)\,.
$$
So each our projection can be estimated as follows
\begin{equation}
\label{starstar}
\|\mathbb{P}_{\nu, Q_{A}} H_{\mu}(\chi_{B\setminus A})\|_{\nu}^2\leq (P_A(\chi_{B\setminus A})\,d\mu)^2 \sum_{J\, \text{good}, J\subset
A'}
\nu(J)\frac{|J|}{|A|} \,.
\end{equation}
So $\|\bP_{\nu, A'} (H_{\mu}\chi_{B\setminus A}\|_{\nu}^2 $ is bounded by
$$
 (P_A(\chi_{B\setminus A})\,d\mu)^2 \sum_{t=j}^{\infty}\sum_{|J|=2^{-t}|A|, J\subset
A}
\nu(J)\frac{|J|}{|A|} \,.
$$
which proves the lemma.
\end{proof}

We first establish a Carleson property for $\{a_S\}$.  
Let $I$ be in $\mathcal{D}_{\mu}$. We choose first the smallest stopping interval containing (it might be equal to) $I$. We call it $\hat{S}$ abusing the notations slightly.Consider the family of its stopping sons $\{S_{\al}\}_{\al \in A}$ such that $S_{\al}\subset  I$. Using our notations for father in the stopping tree $\SSS$ we can write
$$
\hat{S_{\al}} = \hat{S}\,\,\,\forall \al \in A\,.
$$
There can be  a case that such family consists of one interval (call it $S_0$) and $S_0=I$. Consider this case later. Now we assume that all $S_{\al}, al\in A$ are strictly smaller than $I$, and therefore
$$
F(S_{\al})\subset\hat{S}\,\,\,\forall \al \in A\,.
$$
Notice that
\begin{equation}
\label{goodwithF}
 (P_S(\chi_{\hat{S}\setminus F(S_{\al})})\,d\mu)^2\nu(F(S_{\al}))\leq K\mu(F(S_{\al}))\,\,\,\forall \al\in A\,.
 \end{equation}
 But this is not true with replacing $F(S_{\al})$ by $S_{\al}$!
 Let us use naively \eqref{goodwithF} and Lemma \ref{projectionviaPoisson}. Then we get
 $$
  \sum_{\al\in A} \|\bP_{\nu, S_{\al}} (H_{\mu}\chi_{\hat{S}\setminus S_{\al}}\|_{\nu}^2 \leq 2  \sum_{\al\in A} \|\bP_{\nu, S_{\al}} (H_{\mu}\chi_{\hat{S}\setminus F(S_{\al})}\|_{\nu}^2 + 2 \sum_{\al\in A}  \|\bP_{\nu, S_{\al}} (H_{\mu}\chi_{F(S_{\al})\setminus S_{\al}}\|_{\nu}^2 \leq
$$
$$
 2K \sum_{\al\in A}\mu(F(S_{\al})) +2 \sum_{\al\in A}\|(H_{\mu}\chi_{F(S_{\al}}\|_{\nu}^2 \leq (2K +C_{\chi}) \sum_{\al\in A}\mu(F(S_{\al})) \,.
 $$
 In other words we would like to conclude that
\begin{equation}
\label{againstar}
\sum_{\al\in A} \|\mathbb{P}_{\nu, Q_{S_{\al}}} H_{\mu}(\chi_{\hat{S}\setminus S_{\al}})\|_{\nu}^2
\leq C\, \mu(I)\,.
\end{equation}
 But instead, by naive reasoning we achieved
 \begin{equation}
\label{falseagainstar}
\sum_{\al\in A} \|\mathbb{P}_{\nu, Q_{S_{\al}}} H_{\mu}(\chi_{\hat{S}\setminus S_{\al}})\|_{\nu}^2
\leq C\,\sum_{\al\in A}\mu(F(S_{\al}))\,.
\end{equation}
 This is a dangerous place because while the intervals $S_{\al}$ are pairwise disjoint, there fathers
 $F(S_{\al}$'s are usually not and we cannot deduce \eqref{againstar} from \eqref{falseagainstar}, as this is not guaranteed that
 $$
 \sum_{\al\in A}\mu(F(S_{\al})) \leq C\mu(I)\,.
 $$
 This actually is usually false.
 
 However, \eqref{againstar} is true. But the way to prove it is more subtle. Let us do it.
 Let $\{F_{\beta}\}_{\beta\in B}$ denote the family of {\it maximal} intervals among $\{F(S_{\al})\}_{\al\in A}$.
 Let for a given $\beta \in B$ the family $\{S_{\beta,\gamma}\}$ denote all intervals from $\{S_{\al}\}_{\al\in A}$ that lie in $F_{\beta}$. Now
 $$
\sum_{\al\in A} \|\bP_{\nu, S_{\al}} (H_{\mu}\chi_{\hat{S}\setminus S_{\al}})\|_{\nu}^2 =
 \sum_{\beta\in B} \sum_{\gamma}\|\bP_{\nu, S_{\beta,\gamma}} (H_{\mu}\chi_{\hat{S}\setminus S_{\beta,\gamma}})\|^2_{\nu}\leq
$$
$$
2\sum_{\beta\in B} \sum_{\gamma}\|\bP_{\nu, S_{\beta,\gamma}} (H_{\mu}\chi_{\hat{S}\setminus F_{\beta}})\|^2_{\nu} + 2 \sum_{\beta\in B} \sum_{\gamma}\|\bP_{\nu, S_{\beta,\gamma}} (H_{\mu}\chi_{F_{\beta}\setminus S_{\beta,\gamma} })\|^2_{\nu}=: \Sigma_1 +\Sigma_2\,.
$$
For the second sum: $\sum_{\gamma}\|\bP_{\nu, S_{\beta,\gamma}} (H_{\mu}\chi_{F_{\beta}\setminus S_{\beta,\gamma} })\|^2_{\nu}\leq 2\sum_{\gamma}\|\bP_{\nu, S_{\beta,\gamma}} (H_{\mu}\chi_{F_{\beta}})\|_{\nu}^2 + 2\sum_{\gamma}
 \| H_{\mu}\chi_{S_{\beta,\gamma} }\|^2_{\nu}\leq C_{\chi}\mu(F_{\beta})$
by our Sawyer's type test assumption \eqref{Cchi1}. Also we can use now the disjointness of $F_{\beta}$ to conclude that 
$$
\Sigma_2 \leq C\,\mu(I)\,.
$$
For the first sum we use Lemma \ref{projectionviaPoisson} to conclude
$$
\Sigma_1 \leq \sum_{\beta\in B} \sum_{\gamma} (P_{F_{\beta}}\chi_{\hat{S}\setminus F_{\beta}}d\mu)^2\nu(S_{\beta,\gamma} )\leq 
 \sum_{\beta\in B} \sum_{\gamma} (P_{F_{\beta}}\chi_{\hat{S}\setminus F_{\beta}}d\mu)^2 \nu(F_{\beta}) \leq K \sum_{\beta\in B} \mu(F_{\beta})\leq K\,\mu(I)\,.
$$
We used  here the disjointness twice.

Finally \eqref{againstar} is proved. But to prove the estimate of Carleson type for $\{a_S\}_{S\in\SSS}$ we need
not just \eqref{againstar} but
\begin{equation}
\label{againstarall}
\sum_{S\in\SSS, F(S)\subset  I} \|\mathbb{P}_{\nu, Q_{S}} H_{\mu}(\chi_{\hat{S}\setminus S})\|_{\nu}^2
\leq C\, \mu(I)\,.
\end{equation}
We estimated not the whole sum above but only the sum over {\it maximal} $S$ such that $S\in\SSS, F(S)\subset  I$. By the way now it is time to return to the last case: when $S_0 =I$ (see above). Notice that in this case we also estimated
\begin{equation}
\label{againstarmax}
\sum_{S_{\al}\in\SSS, F(S_{\al})\subset  I, S_{\al} \,\text{is maximal}} \|\mathbb{P}_{\nu, Q_{S_{\al}}} H_{\mu}(\chi_{\hat{S_{\al}}\setminus S_{\al}})\|_{\nu}^2
\leq C\, \mu(I)\,.
\end{equation}
But the standard reasoning shows that \eqref{againstarmax} is enough to prove \eqref{againstarall}!
In fact, if our $S$ in the sum in \eqref{againstarall} is not maximal it is contained in a maximal one.
Denoting by $S_j(\al)$ the maximal such $S$ contained in $S_{\al}$ we conclude
$$
\sum_j  \|\mathbb{P}_{\nu, Q_{S_j(\al)}} H_{\mu}(\chi_{\widehat{S_j(\al)}\setminus S_j(\al)})\|_{\nu}^2
\leq C\mu(S_{\al})\,.
$$
We sum over $j$ and $\al$ and notice that our main stopping property says
$$
\sum_{\al} \mu(S_{\al}) \leq  \mu(I)\,.
$$
This gives the sum over maximal intervals inside maximal intervals.
Next generation of stopping intervals will give a contribution $ \frac12 \mu(I)$ because 
$$
\sum_{\al}\sum_j\mu(S_j(\al))\leq \frac12\sum_{\al} \mu(S_{\al})\leq \frac12 \mu(I)\,,
$$
yet next generation will come with the contribution  $ \frac14 \mu(I)$ et cetera...
All this is because of Theorem \ref{yellowS}.
And we obtain  \eqref{againstarall}.

\noindent This gives
\begin{equation}
\label{F0est}
DP=F_0 \leq C\,\|f\|_{\mu}^2\,.
\end{equation}
\noindent We are left to estimate $ODP$.

\subsection{Miraculous improvement of the Carleson property of the sequence $\{a^j_S\}_{S\in\SSS}$}
\label{Miracle}

We used Lemma \ref{projectionviaPoisson} above. But we used it only with $j=0$.
Now we will be estimating Carleson constant for $\{a_S^j\}_{S\in\SSS}$ and it should be exponentially small.
We will use again Lemma \ref{projectionviaPoisson}  but with $j>0$. Recall that $r(S',S)$ denote the tree distance between these two intervals  {\it inside the stopping tree}.  We again consider $I\in \mathcal{D}_{\mu}$, the smallest $\hat{S}\in \SSS$ containing $I$. We need now the estimate
\begin{equation}
\label{againstarall}
\sum_{S\in\SSS, F(S)\subset  I}\sum_{S'\subset S, r(S',S) =j} \|\mathbb{P}_{\nu, Q_{S'}} H_{\mu}(\chi_{\hat{S}\setminus S})\|_{\nu}^2
\leq C\,2^{-cj}\, \mu(I)\,.
\end{equation}

We repeat verbatim the reasoning of the previous section, and of course $2^{-j}$ appears naturally
from Lemma \ref{projectionviaPoisson}. We just use the fact that intervals $S'$ involved in $\mathbb{P}_{\nu, Q_{S'}} $  have the property
$$
t(S',S)\geq  r(S', S) \geq j\,.
$$

The only place where one should be careful to get the extra $2^{-cj}$  is the estimate of $\Sigma_2$.
We cannot use 
$$
\sum_{\gamma}\sum_{S'\subset S_{\beta,\gamma}, r(S',S_{\beta,\gamma})=j}\|\bP_{\nu, S'} (H_{\mu}\chi_{F_{\beta}\setminus S_{\beta,\gamma} })\|^2_{\nu}\leq 
$$
$$
2\sum_{\gamma} \sum_{S'\subset S_{\beta,\gamma}, r(S',S_{\beta,\gamma})=j}\|\bP_{\nu, S'} (H_{\mu}\chi_{F_{\beta}})\|_{\nu}^2
+2\sum_{\gamma}\|
H_{\mu}\chi_{S_{\beta,\gamma} }\|^2_{\nu}\leq C_{\chi}\mu(F_{\beta})
$$ 
anymore. Actually we can say that but this does not give extra $2^{-cj}$.  Instead, by Lemma \ref{projectionviaPoisson}
$$
\sum_{\gamma}\sum_{S'\subset S_{\beta,\gamma}, r(S',S_{\beta,\gamma})=j}\|\bP_{\nu, S'} (H_{\mu}\chi_{F_{\beta}\setminus S_{\beta,\gamma} })\|^2_{\nu}\leq 
$$
$$
C\,2^{-j}\sum_{\gamma} \sum_{S'\subset S_{\beta,\gamma}, r(S',S_{\beta,\gamma})=j}\nu(S')
(P_{S_{\beta, \gamma}}\chi_{F_{\beta} \setminus S_{\beta,\gamma}}d\mu)^2\leq
$$
$$
C\,2^{-j}\sum_{\gamma}\nu(S_{\beta,\gamma})(P_{S_{\beta, \gamma}}\chi_{F_{\beta} \setminus S_{\beta,\gamma}}d\mu)^2
$$
For a fixed $\beta$, the intervals $S_{\beta,\gamma}$ are disjoint by their construction (see above).
It is time to use \eqref{PIVOTAL}. (We use it here for the second time in our proof, the first one was in Theorem \ref{yellowS}.)  If we apply \eqref{PIVOTAL} to the last sum, we get
$$
\sum_{\gamma}\nu(S_{\beta,\gamma})(P_{S_{\beta, \gamma}}\chi_{F_{\beta} \setminus S_{\beta,\gamma}}d\mu)^2\leq
P\, \mu(F_{\beta})\,.
$$
Therefore,
 $$
 \sum_{\gamma}\sum_{S'\subset S_{\beta,\gamma}, r(S',S_{\beta,\gamma})=j}\|\bP_{\nu, S'} (H_{\mu}\chi_{F_{\beta}\setminus S_{\beta,\gamma} })\|^2_{\nu}\leq
c\,2^{-j}\mu(F_{\beta})\,.
$$
We already said that all terms, in particular, the  analog of the sum $\Sigma_1$ also get $2^{-j}$ factor.
This is nice as we get
$$
\sum_{S\in\SSS, F(S)\subset I, S \,\text {is maximal}} a_S^j \leq c\,2^{-j} \mu(I)\,.
$$
Now we again need to estimates the whole sum
\begin{equation}
\label{againstarj}
\sum_{S\in\SSS, F(S)\subset I} a_S^j \leq c\,2^{-j} \mu(I)\,.
\end{equation}
This achieved exactly as before with the help of \eqref{yellowCarl} of Theorem \ref{yellowS}.
We consider $S_{\al}$ to be maximal $S\in\SSS, F(S)\subset I$, and then for a fixed $\al$ consider 
$S_j(\al)$ to be maximal $S\in\SSS, F(S)\subset S_{\al}$. 

Next generation of stopping intervals will give a contribution $ \frac12 2^{-j}\mu(I)$ because 
$$
\sum_{\al}\sum_j\mu(S_j(\al))\leq \frac12\sum_{\al} \mu(S_{\al})\leq \frac12 \mu(I)\,,
$$
yet next generation will come with the contribution  $ \frac14 2^{-j}\mu(I)$ et cetera... And we get \eqref{againstarj}.
All this is because of Theorem \ref{yellowS}.

Theorem \ref{MainTwoWeight} is completely proved.

\end{document}